\documentclass{siamltex}
\usepackage{elf}
\usepackage{subfigure}
\usepackage{geometry}

\newtheorem{remark}[theorem]{Remark}

\title{A Finite Element Discretization of the Streamfunction Formulation of the
Stationary Quasi-Geostrophic Equations of the Ocean}

\author{Erich L Foster 
\and Traian Iliescu 
\and Zhu Wang \thanks{Department of Mathematics, Virginia Tech, McBryde Hall, Blacksburg, VA 24061-0123, USA}}

\begin{document}
  \maketitle

  \begin{abstract}
    This paper presents a conforming finite element discretization of the streamfunction formulation
    of the one-layer stationary quasi-geostrophic equations, which are a commonly used model for the
    large scale wind-driven ocean circulation. Optimal error estimates for this finite element
    discretization with the Argyris element are derived. Numerical tests for the finite element
    discretization of the quasi-geostrophic equations and two of its standard simplifications (the
    linear Stommel model and the linear Stommel-Munk model) are carried out. By benchmarking the
    numerical results against those in the published literature, we conclude that our finite element
    discretization is accurate. Furthermore, the numerical results have the same convergence rates as
    those predicted by the theoretical error estimates.
  \end{abstract}


  \section{Introduction}

With the continuous increase in computational power, complex mathematical
models are becoming more and more popular in the numerical simulation of
oceanic and atmospheric flows. For some geophysical flows in which computational
efficiency is of paramount importance, however, simplified mathematical models
are central. 
For example, the \emph{quasi-geostrophic equations (QGE)}, a standard mathematical model for large scale oceanic and atmospheric flows \cite{cushman2011introduction,Majda,Ped92,Vallis06}, are often used in climate modeling \cite{Dijkstra05}.

The QGE are usually discretized in space by using the \emph{finite difference
method} (FDM) \cite{San11}. The \emph{finite element method} (FEM), however, offers several
advantages over the popular FDM, as outlined in \cite{Myers}: (i) an easy treatment of
complex boundaries, such as those of continents for the ocean, or mountains for
the atmosphere; (ii) an easy grid refinement to achieve a high resolution in
regions of interest \cite{Cascon}; (iii) a natural treatment of boundary conditions; and
(iv) a straightforward approach for the treatment of multiply connected domains
\cite{Myers}. Despite these advantages, there are relatively few papers that consider
the FEM applied to the QGE \cite{Cascon,Fix,leprovost1994comparison,Myers,stevens1982finite}.

To our knowledge, {\it all} the {\it finite element (FE)} discretizations of the QGE have been developed for the streamfunction-vorticity formulation, none using the streamfunction formulation.
The reason is simple: The streamfunction-vorticity formulation yields a second order {\it partial differential equation (PDE)}, whereas the streamfunction formulation yields a fourth order PDE.
Thus, although the streamfunction-vorticity formulation has two variables ($q$ and $\psi$) and the streamfunction formulation has just one ($\psi$), the former is the preferred formulation used in practical computations, since its conforming FE discretization requires low-order ($C^0$) elements, whereas the latter requires high-order ($C^1$) elements.

Although the FE discretizations of the QGE are relatively scarce, the corresponding error analysis seems to be even more scarce.
To our knowledge, {\it all} the error analysis for the FE discretization of the QGE has been done for the streamfunction-vorticity formulation, and none has been done for the streamfunction formulation.
Furthermore, to the best of our knowledge, all the available error estimates for the FE discretization of the QGE are {\it suboptimal}.
The first error analysis for the FE discretization of the QGE was carried out by Fix \cite{Fix}, in which suboptimal error estimates for the streamfunction-vorticity formulation were proved.
Indeed, relationships (4.7) and (4.8) (and the discussion above these) in \cite{Fix} show that the FE approximations for {\it both} the potential vorticity (denoted by $\zeta$) and streamfunction (denoted by $\psi$) consist of piecewise polynomials of degree $k-1$.
At the top of page 381, the author concludes that the error analysis yields the following estimates:
\begin{eqnarray}
\| \psi - \psi^h \|_1
&=& O(h^{k-1}) ,
\label{eqn:fix_1} \\
\| \zeta - \zeta ^h \|_0
&=& O(h^{k-1}) .
\label{eqn:fix_2}
\end{eqnarray}
Although the streamfunction error estimate \eqref{eqn:fix_1} appears to be optimal, the potential vorticity error estimate \eqref{eqn:fix_2} is clearly suboptimal.
Indeed, using piecewise polynomials of degree $k-1$ for the FE approximation of the vorticity, one would expect an $O(h^k)$ error estimate in the $L^2$ norm.
Medjo \cite{Medjo99, Medjo00} used a FE discretization of the streamfunction-vorticity formulation and proved error estimates for the time discretization, but no error estimates for the spatial discretization.
Finally, Cascon {\it et al.} \cite{Cascon} proved both {\it a priori} and {\it a posteriori} error estimates for the FE discretization of the {\it linear Stommel-Munk} model (see Section~\ref{sec:Tests} for more details).
This model, while similar to the QGE, has one significant difference: the linear Stommel-Munk model is linear, whereas the QGE are nonlinear.

We note that the state-of-the-art in the FE error analysis for the QGE seems to reflect that for the {\it two-dimensional Navier-Stokes equations (2D NSE)}, to which the QGE are similar in form. 
Indeed, as carefully discussed in \cite{Gunzburger89} (see also \cite{Fairag98,Fairag03,gunzburger1988finite,gunzburger1988onfinite}), the 2D NSE in streamfunction-vorticity formulation are easy to implement (only $C^0$ elements are needed for a conforming discretization), but the available error estimates are suboptimal (see Section 11.6 in \cite{Gunzburger89}).
Next, we summarize the discussion in \cite{Gunzburger89}, since we believe it sheds light on the QGE setting.
For $C^0$ piecewise polynomial of degree $k$ FE approximation for {\it both} the vorticity (denoted by $\omega$) and streamfunction (denoted by $\psi$), the error estimates given in \cite{Girault86} are (see (11.26) in \cite{Gunzburger89}):
\begin{eqnarray}
| \psi - \psi^h |_1
+ \| \omega - \omega^h \|_0
\leq C \, h^{k - 1/2} \, | \ln h |^{\sigma} ,
\label{eqn:gunzburger_1}
\end{eqnarray}
where $\sigma = 1$ for $k = 1$ and $\sigma = 0$ for $k > 1$.
It is noted in \cite{Gunzburger89} that the error estimate in \eqref{eqn:gunzburger_1} is not optimal: one may loose a half power in $h$ for the derivatives of the streamfunction (i.e., for the velocity), and three-halves power for the vorticity.
It is also noted that there is computational and theoretical evidence that \eqref{eqn:gunzburger_1} is not sharp with respect to the streamfunction error.
Furthermore, in \cite{fix1984mixed} it was shown that, for the {\it linear} Stokes equations, the derivatives of the streamfunction are essentially optimally approximated (see (11.27) in \cite{Gunzburger89}):
\begin{eqnarray}
| \psi - \psi^h |_1
\leq C \, h^{k - \varepsilon} ,
\label{eqn:gunzburger_2}
\end{eqnarray}
where $\varepsilon = 0$ for $k > 1$ and $\varepsilon > 0$ is arbitrary for $k = 1$.
It is, however, noted in \cite{Gunzburger89} that  \eqref{eqn:gunzburger_1} seems to be sharp for the vorticity error and thus vorticity approximations are generally poor.

The streamfunction formulation is, from both mathematical and computational points of view, completely different from the streamfunction-vorticity formulation.
Indeed, the FE discretization of the streamfunction formulation generally requires the use of $C^1$ elements (for a conforming discretization), which makes their implementation challenging.
From a mathematical point of view, however, the streamfunction formulation has the following significant advantage over the streamfunction-vorticity formulation: there are optimal error estimates for the FE discretization of the streamfunction formulation (see the error estimate (13.5) and Table 13.1 in \cite{Gunzburger89}), whereas the available error estimates for the streamfunction-vorticity formulation are suboptimal.

The main goal of this paper is twofold. 
First, we use a $C^1$ finite element (the Argyris element) to discretize the streamfunction formulation of the QGE. 
To the best of our knowledge, this is the {\it first} time that a $C^1$ finite element has been used in the numerical discretization of the QGE. 
Second, we derive optimal error estimates for the FE discretization of the QGE and present supporting numerical experiments.
To the best of our knowledge, this is the {\it first} time that optimal error estimates for the QGE have been derived.

The rest of the paper is organized as follows: Section \ref{sec:QGE} presents the QGE, their weak
formulation, and mathematical support for the weak formulation. Section \ref{sec:FEM} outlines the FEM
discretization of the QGE, posing a special emphasis of the Argyris element. Rigorous error
estimates for the FE discretization of the stationary QGE are derived in Section \ref{sec:Errors}.
Several numerical experiments supporting the theoretical results are presented in
Section \ref{sec:NumericalResults}. Finally, conclusions and our future research directions are included in
Section \ref{sec:Conclusions}.
  \section{The Quasi-Geostrophic Equations} \label{sec:QGE}
    
The large scale ocean flows, which play a significant role in climate dynamics \cite{Dijkstra05,ghil2008climate}, are driven by two major sources: the wind and the buoyancy (see, e.g., Chapters 14-16 in \cite{Vallis06}).  Winds drive the
subtropical and subpolar gyres, which correspond to the strong, persistent, subtropical and
subpolar western boundary currents in the North Atlantic Ocean (the Gulf Stream and the Labrador
Current) and North Pacific Ocean (the Kuroshio and the Oyashio Currents), as well as their
subtropical counterparts in the southern hemisphere \cite{Dijkstra05,Vallis06}.  One of the common
features of these gyres is that they display strong western boundary currents, weak interior
flows, and weak eastern boundary currents.  

One of the most popular mathematical models used in the study of large scale wind-driven ocean
circulation is the QGE \cite{cushman2011introduction,Vallis06}.  The QGE represent a simplified model of
the full-fledged equations (e.g., the Boussinesq equations), which allows efficient numerical
simulations while preserving many of the essential features of the underlying large scale ocean
flows.  The assumptions used in the derivation of the QGE include the hydrostatic balance, the
$\beta$-plane approximation, the geostrophic balance, and the eddy viscosity parametrization.
Details of the derivation of the QGE and the approximations used along the way can be found in
standard textbooks on geophysical fluid dynamics, such as
\cite{cushman2011introduction,Maj03,Majda,mcwilliams2006fundamentals,Ped92,Vallis06}. 

In the {\it one-layer QGE}, sometimes called the barotropic vorticity equation, the flow is assumed to be homogenous in the vertical direction. 
Thus, stratification effects are ignored in this model. 
The practical advantages of such a choice are obvious: the computations are two-dimensional, and, thus, the corresponding numerical simulation have a low computational cost.
To include stratification effects, QGE models of increasing complexity have been devised by increasing the number of layers in the model (e.g., the two-layer QGE and the $N$-layer QGE \cite{Vallis06}).
As a first step, in this report we use the one-layer QGE (referred to as ``the QGE" in what follows) to study the wind-driven circulation in an enclosed, midlatitude rectangular basin, which is a standard problem, studied extensively by ocean modelers \cite{cushman2011introduction,Maj03,Majda,mcwilliams2006fundamentals,Ped92,Vallis06}.

The nondimensional {\it streamfunction-vorticity formulation} of the {\it stationary one-layer quasi-geostrophic equations} is (see, e.g., equation (14.57) in \cite{Vallis06}, equation (1.1) in \cite{Majda}, equation (1.1) in \cite{wang1994emergence}, and equation (1) in \cite{greatbatch2000four}):
\begin{eqnarray}
J(\psi , q)
&=& - Re^{-1} \, \Delta q
+ F
\label{qge_q_psi_1} \\
q
&=& - Ro \, \Delta \psi
+ y ,
\label{qge_q_psi_2}
\end{eqnarray}
where 
$\psi$ is the velocity streamfunction,
$q$ is the potential vorticity,
$F$ is the forcing,
$J(\cdot , \cdot)$ is the Jacobian operator given by
\begin{eqnarray}
J(\psi , q)
:= \frac{\partial \psi}{\partial x} \, \frac{\partial q}{\partial y}
-   \frac{\partial \psi}{\partial y} \, \frac{\partial q}{\partial x} ,
\label{eqn:jacobian}
\end{eqnarray}
$Re$ is the Reynolds number, and
$Ro$ is the Rossby number.
The Rossby number, $Ro$, is defined as
\begin{eqnarray}
Ro
:= \frac{U}{\beta \, L^2} ,
\label{eqn:rossby_number}
\end{eqnarray}
where 
$\beta$ is the coefficient multiplying the $y$ coordinate in the $\beta$-plane approximation \cite{cushman2011introduction,Vallis06}, 
$L$ is the width of the computational domain, and
$U$ is the Sverdrup velocity obtained from the balance between the $\beta$-effect and the curl of the divergence of the wind stress \cite{Vallis06}.
The Reynolds number, $Re$, is defined as 
\begin{eqnarray}
Re
:= \frac{U \, L}{A} ,
\label{eqn:reynolds_number}
\end{eqnarray}
where $A$ is the eddy viscosity parametrization.
The horizontal velocity ${\bf u}$ can be recovered from $\psi$ by using the following formula:
$
{\bf u}
= \left(
\frac{\partial \psi}{\partial y} ,
- \frac{\partial \psi}{\partial x}
\right) .
$

Substituting \eqref{qge_q_psi_2} in \eqref{qge_q_psi_1} and dividing by $Ro$, we get the {\it streamfunction formulation} of the {\it stationary one-layer quasi-geostrophic equations}
\begin{eqnarray}
Re^{-1} \, \Delta^2 \psi 
+ J(\psi , \Delta \psi)
- Ro^{-1} \, \frac{\partial \psi}{\partial x}
= Ro^{-1} \, F .
\label{qge_psi_1}
\end{eqnarray}

We note that the streamfunction-vorticity formulation has two unknowns ($q$ and $\psi$), whereas the streamfunction formulation has only one unknown ($\psi$).
Because the streamfunction-vorticity formulation  is a second-order PDE, whereas the streamfunction formulation is a fourth-order PDE, the former is more popular in practical computations.

We also note that \eqref{qge_q_psi_1}-\eqref{qge_q_psi_2} and \eqref{qge_psi_1} are similar in form to the 2D NSE written in the streamfunction-vorticity and streamfunction formulations, respectively.
There are, however, several significant differences between the QGE and the 2D NSE.
First, the term $y$ in \eqref{qge_q_psi_2} and the corresponding term $\frac{\partial \psi}{\partial x}$ in \eqref{qge_psi_1}, which model the {\it rotation effects} in the QGE, do not have counterparts in the 2D NSE.
Second, the Rossby number, $Ro$, in the QGE, which is a measure of the rotation effects, does not appear in the 2D NSE.

Next, we comment on the significance of the two parameters in \eqref{qge_psi_1}, the Reynolds number, $Re,$ and the Rossby number, $Ro$.
As in the 2D NSE case, $Re$ is the coefficient of the diffusion term $- \Delta q = \Delta^2 \psi$.
The higher the Reynolds number $Re$, the smaller the magnitude of the diffusion term as compared with the nonlinear convective term $J(\psi, \Delta \psi)$.
For small $Ro$, which corresponds to large rotation effects, the forcing term, $Ro^{-1} \, F$, becomes large compared with the other terms.
The term $Ro^{-1} \, \frac{\partial \psi}{\partial x}$ could be interpreted as a convection type term with respect to $\psi$, not to $q = -\Delta \psi$.
When $Ro$ is small, $Ro^{-1} \, \frac{\partial \psi}{\partial x}$ becomes large. 
Thus, the physically relevant cases for large scale oceanic flows, in which $Re$ is large and $Ro$ is small (i.e., small diffusion and high rotation, respectively) translate mathematically into a {\it convection-dominated} PDE with {\it large forcing}.
Thus, from a mathematical point of view, we expect the restrictive conditions used to prove the well-posedness of the 2D NSE \cite{Girault79,Girault86,Gunzburger89} to be even more restrictive in the QGE setting, due to the rotation effects.
We will later see that this is indeed the case.

To completely specify the equations in \eqref{qge_psi_1}, we need to impose boundary conditions. 
The question of appropriate boundary conditions for the QGE is a thorny one, especially for the streamfunction-vorticity formulation (see, e.g., \cite{Cummins,Vallis06}).
In this report, we consider $\psi = \frac{\partial \psi}{\partial {\bf n}}= 0$ on $\partial \Omega$, which are also used in \cite{Gunzburger89} for the streamfunction formulation of the 2D NSE.

To derive the weak formulation of the QGE \eqref{qge_psi_1}, we first introduce the appropriate functional setting.
Let
$  X 
  := H^2_0(\Omega) 
  = \left\{ \psi\in H^2(\Omega):
    \psi=\frac{\partial\psi}{\partial {\bf n}}=0 \text{ on } \partial\Omega \right\}
$.
Multiplying \eqref{qge_psi_1} by a test function $\chi \in X$ and using the divergence theorem, we get the {\it weak formulation} of the QGE in streamfunction formulation \cite{Gunzburger89}:
\begin{eqnarray}
&& Re^{-1} \, \int_{\Omega} \, \Delta \psi \, \Delta \chi \, d{\bf x}
+ \int_{\Omega} \, \Delta \psi \, \left( \psi_y \, \chi_x - \psi_x \, \chi_y \right) \, d{\bf x}
- Ro^{-1} \, \int_{\Omega} \, \psi_x \, \chi \, d{\bf x} 
\nonumber \\
&& \hspace*{3.0cm} = Ro^{-1} \, \int_{\Omega} \, F \, \chi \, d{\bf x} 
\qquad \forall \, \chi \in X .
\label{eqn:qge_psi_weak}
\end{eqnarray}
Therefore, letting
\begin{align}
  a_0(\psi,\chi) &= Re^{-1} \, \int_{\Omega} \, \Delta \psi \, \Delta \chi \, d{\bf x} , \label{eqn:a0} \\
  a_1(\zeta,\psi,\chi) &= \int_{\Omega} \, \Delta \zeta \, \left( \psi_y \, \chi_x - \psi_x \, \chi_y \right) \, d{\bf x} , \label{eqn:a1} \\
  a_2(\psi,\chi) &= - Ro^{-1} \, \int_{\Omega} \, \psi_x \, \chi \, d{\bf x} , \label{eqn:a2} \\
  \ell(\chi) &= Ro^{-1} \, \int_{\Omega} \, F \, \chi \, d{\bf x} , \label{eqn:l}
\end{align}
gives the weak formulation of the QGE in streamfunction formulation: Find $\psi \in X$ such that
\begin{equation}
    a_0(\psi,\chi) + a_1(\psi,\psi,\chi) + a_2(\psi,\chi) = \ell(\chi),\quad \forall
      \chi \in X.
    \label{eqn:WeakForm}
\end{equation}
The linear form $\ell$, the bilinear forms $a_0$ and $a_2$, and the trilinear
form $a_1$ are continuous: There exist $\Gamma_1 > 0$ and $\Gamma_2 > 0$ such that 
\begin{align}
  | a_0(\psi,\chi) | &\le Re^{-1} \,  |\psi|_2 \, |\chi|_2 \quad \forall \, \psi,\chi\in X , \label{eqn:a0Cont} \\
  | a_1(\zeta,\psi,\chi) | &\le \Gamma_1 \, |\zeta|_2 \, |\psi|_2 \, |\chi|_2 \quad \forall \, \zeta,\psi,\chi\in X , \label{eqn:a1cont} \\
  | a_2(\psi,\chi) | &\le Ro^{-1} \, \Gamma_2 \, |\psi|_2 \, |\chi|_2 \quad \forall \, \psi, \, \chi \in X , \label{eqn:a2Cont} \\
  | \ell(\chi) | &\le Ro^{-1} \,\|F\|_{-2} \, |\chi|_2 \quad \forall \, \chi \in X . \label{eqn:lCont}
\end{align}
Inequalities \eqref{eqn:a0Cont}, \eqref{eqn:a1cont}, and \eqref{eqn:lCont} are stated in \cite{Cayco86} (see inequalities (2.2) and (2.3) in  \cite{Cayco86}).
Inequality \eqref{eqn:a2Cont} can be proved as follows.
Proposition 2.1(iii) in \cite{Medjo99} implies that 
\begin{eqnarray}
| a_2(\psi,\chi) | 
\le Ro^{-1} \,C \, \|\psi\|_2 \, \|\chi\|_2 ,
\label{eqn:medjo}
\end{eqnarray}
where $C$ is a generic constant.
Theorem 1.1 in \cite{Girault86} implies that $| \cdot |_2$, the $H^2$-seminorm, and $\| \cdot \|_2$, the $H^2$-norm are equivalent on $X = H_0^2$.
Thus, \eqref{eqn:medjo} yields inequality \eqref{eqn:a2Cont}.

For small enough data, one can use the same type of arguments as in \cite{Girault79,Girault86} to prove that the QGE in streamfunction formulation \eqref{eqn:WeakForm} are well-posed \cite{barcilon1988existence,wolansky1988existence}.
In what follows, we will always assume that the small data condition involving $Re$, $Ro$ and $F$, is satisfied and, thus, that there exists a unique solution $\psi$ to \eqref{eqn:WeakForm}.

Using a standard argument \cite{Cayco86}, one can also prove the following stability estimate:
\begin{theorem} \label{thm:stability}
The solution $\psi$ of \eqref{eqn:WeakForm} satisfies the following stability estimate:
 \begin{equation}
   |\psi|_2 
   \le Re \, Ro^{-1} \, \| F \|_{-2} . 
   \label{eqn:stability_1}
 \end{equation}
\end{theorem}
\begin{proof}
Setting $\chi = \psi$ in \eqref{eqn:WeakForm}, we get:
\begin{eqnarray}
a_0(\psi, \psi) 
+ a_1(\psi,\psi, \psi) 
+ a_2(\psi, \psi) 
= \ell(\psi) .
\label{eqn:stability_2}
\end{eqnarray}
Since the trilinear form $a_1$ is skew-symmetric in the last two arguments \cite{Girault79,Girault86,Gunzburger89}, we have
\begin{eqnarray}
a_1(\psi,\psi, \psi) 
= 0 .
\label{eqn:stability_3}
\end{eqnarray}
We also note that, applying Green's theorem, we have
\begin{eqnarray}
a_2(\psi,\psi) 
&=& - Ro^{-1} \, \iint_{\Omega} \frac{\partial \psi}{\partial x} \, \psi \, dx \, dy 
\, = \, - \frac{Ro^{-1}}{2} \, \iint_{\Omega} \frac{\partial}{\partial x} (\psi^2) \, dx \, dy \nonumber \\
&=& - \frac{Ro^{-1}}{2} \, \iint_{\Omega} \left( \frac{\partial}{\partial x} (\psi^2) 
- \frac{\partial}{\partial y} (0) \right) \, dx \, dy 
\, = \, - \frac{Ro^{-1}}{2} \, \int_{\partial \Omega} 0 \, dx + \psi^2 \, dy 
= 0 ,
\label{eqn:stability_4}
\end{eqnarray}
where in the last equality in \eqref{eqn:stability_4} we used that $\psi = 0$ on $\partial \Omega$ (since $\psi \in H_0^2(\Omega)$).
Substituting \eqref{eqn:stability_4} and \eqref{eqn:stability_3} in \eqref{eqn:stability_2} and using the Cauchy-Schwarz inequality, we get:
\begin{eqnarray}
| \psi |_{2}^2
= \int_{\Omega} \Delta \psi \, \Delta \psi \, d{\bf x}
= Re \, Ro^{-1} \, \int_{\Omega} F \, \psi \, d{\bf x}
\leq Re \, Ro^{-1} \, \| F \|_{-2} \, | \psi |_{2} ,
\label{eqn:stability_5}
\end{eqnarray}
which proves \eqref{eqn:stability_1}.
\hfill
\end{proof}
  \section{Finite Element Formulation} \label{sec:FEM}
In this section, we present the functional setting and some auxiliary results for the FE discretization of the streamfunction formulation of the QGE \eqref{eqn:WeakForm}.
Let $\mathcal{T}^h$ denote a finite element triangulation of $\Omega$ with meshsize (maximum triangle diameter) $h$.
We consider a {\it conforming} FE discretization of \eqref{eqn:WeakForm}, i.e., $X^h \subset X = H_0^2(\Omega)$.

The FE discretization of the streamfunction formulation of the QGE \eqref{eqn:WeakForm} reads:
Find $\psi^h \in X^h$ such that 
\begin{eqnarray}
    a_0(\psi^h,\chi^h) + a_1(\psi^h,\psi^h,\chi^h) + a_2(\psi^h,\chi^h) = \ell(\chi^h),\quad \forall \, 
      \chi^h \in X^h.
    \label{eqn:FEForm}
\end{eqnarray}
Using standard arguments \cite{Girault79,Girault86}, one can prove that, if the small data condition used in proving the well-posedness result for the continuous case holds, then \eqref{eqn:FEForm} has a unique solution $\psi^h$ (see Theorem 2.1 and subsequent discussion in \cite{Cayco86}).
One can also prove the following stability result for $\psi^h$ using the same arguments as those used in the proof of Theorem \ref{thm:stability} for the continuous setting.
\begin{theorem} \label{thm:stability_fem}
The solution $\psi^h$ of \eqref{eqn:FEForm} satisfies the following stability estimate:
 \begin{equation}
   |\psi^h|_2 
   \le Re \, Ro^{-1} \, \| F \|_{-2} . 
   \label{eqn:stability_fem_1}
 \end{equation}
\end{theorem}

In order to develop a conforming FEM for the QGE \eqref{eqn:WeakForm}, we need to construct subspaces of $H^2_0(\Omega)$, i.e., to find $C^1$ FEs, such as the Argyris triangular element, the Bell triangular element, the Hsieh-Clough-Tocher triangular element (a macroelement), or the Bogner-Fox-Schmit rectangular element \cite{Ciarlet,Gunzburger89,Johnson,braess2001finite}.
%
In what follows, we will use the {\it Argyris FE}.
The Argyris FE employs piecewise polynomials of degree five and has twenty-one {\it degrees of freedom (DOFs)}: the value at each vertex, the value of the first derivatives at each vertex, the value of the second derivatives at each vertex, the value of the mixed derivative at each vertex, and the value of the normal derivatives at each of the edge midpoints.  
To maintain the direction of the normal derivatives in the transformation from the reference element to the physical element, we use the approach developed in \cite{Dominguez06}.
  
By using Theorem 6.1.1 and inequality (6.1.5) in \cite{Ciarlet}, we obtain the following three approximation properties for the Argyris FE space $X^h$:
\begin{eqnarray}
\forall \, \chi \in H^6(\Omega) \cap H^2_0(\Omega), \ \exists \, \chi^h \in X^h
\quad \text{such that} \quad
\| \chi - \chi^h \|_2
&\leq& C \, h^4 \, | \chi |_6 ,
\label{eqn:argyris_approximation_0} \\[0.2cm]
\forall \, \chi \in H^4(\Omega) \cap H^2_0(\Omega), \ \exists \, \chi^h \in X^h
\quad \text{such that} \quad
\| \chi - \chi^h \|_2
&\leq& C \, h^2 \, | \chi |_4 ,
\label{eqn:argyris_approximation_1} \\[0.2cm]
\forall \, \chi \in H^3(\Omega) \cap H^2_0(\Omega), \ \exists \, \chi^h \in X^h
\quad \text{such that} \quad
\| \chi - \chi^h \|_2
&\leq& C \, h \, | \chi |_3 ,
\label{eqn:argyris_approximation_2}
\end{eqnarray}
where $C$ is a generic constant that can depend on the data, but not on the meshsize $h$.
Property \eqref{eqn:argyris_approximation_0} follows from (6.1.5) in \cite{Ciarlet} with $q = 2, \, p = 2, \, m = 2$ and $k+1 = 6$.
Property \eqref{eqn:argyris_approximation_1} follows from (6.1.5) in \cite{Ciarlet} with $q = 2, \, p = 2, \, m = 2$ and $k+1 = 4$.
Finally, property \eqref{eqn:argyris_approximation_2} follows from (6.1.5) in \cite{Ciarlet} with $q = 2, \, p = 2, \, m = 2$ and $k+1 = 3$.

  \section{Error Analysis} \label{sec:Errors}
The main goal of this section is to develop a rigorous numerical analysis for the FE discretization of the QGE \eqref{eqn:FEForm} by using the conforming Argyris element.
In Theorem~\ref{thm:EnergyNorm}, we prove error estimates in the $H^2$ norm by using an approach similar to that used in \cite{Cayco86}.
In Theorem~\ref{thm:Errors}, we prove error estimates in the $L^2$ and $H^1$ norms by using a duality argument.

\begin{theorem}
\label{thm:EnergyNorm}
  Let $\psi$ be the solution of \eqref{eqn:WeakForm} and $\psi^h$ be the solution
  of \eqref{eqn:FEForm}. 
  Furthermore, assume that the following small data condition is satisfied:
  \begin{eqnarray}
  Re^{-2} \, Ro 
  \geq \Gamma_1 \, \| F \|_{-2} ,
  \label{eqn:small_data_condition}
  \end{eqnarray}
where 
$Re$ is the Reynolds number defined in \eqref{eqn:reynolds_number}, 
$Ro$  is the Rossby number defined in \eqref{eqn:rossby_number},
 $\Gamma_1$ is the continuity constant of the trilinear form $a_1$ in \eqref{eqn:a1cont}, and 
 $F$ is the forcing term.
Then the following error estimate holds:
  \begin{equation}
    |\psi - \psi^h|_2 
    \le C(Re, Ro, \Gamma_1, \Gamma_2, F) \, \inf_{\chi^h \in X^h} |\psi - \chi^h|_2 ,
    \label{eqn:EnergyNorm}
  \end{equation}
  where $\Gamma_2$ is the continuity constant of the bilinear form $a_2$ in \eqref{eqn:a2Cont} and 
  \begin{eqnarray} 
  C(Re, Ro, \Gamma_1, \Gamma_2, F)
  := 
  \frac{
   Ro^{-1} \, \Gamma_2
  + 2 \, Re^{-1} 
  + \Gamma_1 \, Re \, Ro^{-1} \, \| F \|_{-2}
  }
  {
  Re^{-1}
  - \Gamma_1 \, Re \, Ro^{-1} \, \| F \|_{-2}
  }
  \label{eqn:constant_definition}
  \end{eqnarray}
is a generic constant that can depend on $Re$, $Ro$, $\Gamma_1$, $\Gamma_2$, $F$, but {\it not} on the meshsize $h$.
\end{theorem}

\begin{remark}
Note that the small data condition in Theorem~\ref{thm:EnergyNorm} involves {\it both} the Reynolds number and the Rossby number, the latter quantifying the rotation effects in the QGE.

Furthermore, note that the standard small data condition $Re^{-2} \geq \Gamma_1 \, \| F \|_{-2}$ used to prove the uniqueness for the steady-state 2D NSE \cite{Girault79,Girault86,layton2008introduction,temam2001navier} is significantly more restrictive for the QGE, since \eqref{eqn:small_data_condition} has the Rossby number (which is small when rotation effects are significant) on the left-hand side.
This is somewhat counterintuitive, since in general rotation effects are expected to help in proving the well-posedness of the system.
We think that the explanation is the following: 
Rotation effects do make the mathematical analysis of 3D flows more amenable by giving them a 2D character.
We, however, are concerned with 2D flows (the QGE).
In this case, the small data condition \eqref{eqn:small_data_condition} (needed in proving the uniqueness of the solution) indicates that rotation effects make the mathematical analysis of the (2D) QGE more complicated than that of the 2D NSE.
\end{remark}
\begin{proof}
  Since $X^h \subset X$, \eqref{eqn:WeakForm} holds for all $\chi = \chi^h\in X^h$.
  Subtracting \eqref{eqn:FEForm} from \eqref{eqn:WeakForm} with $\chi=\chi^h \in
  X^h$ gives
  \begin{equation}
      a_0(\psi - \psi^h,\chi^h) + a_1(\psi,\psi,\chi^h) -
      a_1(\psi^h,\psi^h,\chi^h) 
      + a_2(\psi-\psi^h,\chi^h) = 0 \qquad \forall \chi^h \in
    X^h.
    \label{eqn:ErrorEq}
  \end{equation}
  Next, adding and subtracting $a_1(\psi^h,\psi,\chi^h)$ to \eqref{eqn:ErrorEq}, we get
  \begin{eqnarray}
      && a_0(\psi - \psi^h,\chi^h) + a_1(\psi,\psi,\chi^h) - a_1(\psi^h,\psi,\chi^h) 
      + a_1(\psi^h,\psi,\chi^h) - a_1(\psi^h,\psi^h,\chi^h) \nonumber \\
      && \hspace*{3.0cm} + a_2(\psi-\psi^h,\chi^h) = 0 \qquad \forall \chi^h \in X^h.
    \label{eqn:interError}
  \end{eqnarray}
  The error $e$ can be decomposed as 
    $e:= \psi-\psi^h = (\psi-\lambda^h)+(\lambda^h-\psi^h):= \eta + \varphi^h$,
  where $\lambda^h\in X^h$ is arbitrary. 
  Thus, equation \eqref{eqn:interError} can be
  rewritten as 
  \begin{equation}
      a_0(\eta+\varphi^h,\chi^h)+a_1(\eta+\varphi^h,\psi,\chi^h)+a_1(\psi^h,\eta+\varphi^h,\chi^h) 
      + a_2(\eta+\varphi^h,\chi^h)=0 \qquad \forall \chi^h \in X^h.
    \label{eqn:ErrorEta}
  \end{equation}
  Letting $\chi^h := \varphi^h$ in \eqref{eqn:ErrorEta}, we obtain 
  \begin{equation}
    \begin{split}
      a_0(\varphi^h,\varphi^h) + a_2(\varphi^h,\varphi^h) = -a_0(\eta,\varphi^h)
      - a_1(\eta,\psi,\varphi^h) - a_1(\varphi^h,\psi,\varphi^h) \\
      - a_1(\psi^h,\eta,\varphi^h) - a_1(\psi^h,\varphi^h,\varphi^h)
      -a_2(\eta,\varphi^h).
    \end{split}
    \label{eqn:ErrorVarphi}
  \end{equation}
  Note that, since $a_2(\varphi^h,\varphi^h)=-a_2(\varphi^h,\varphi^h)\; \forall
  \varphi^h \in X^h \subset X = H^2_0$, 
it follows that 
    $a_2(\varphi^h,\varphi^h)=0$.
  We also have that 
    $a_1(\psi^h,\varphi^h,\varphi^h)=0$.
Using these equalities in \eqref{eqn:ErrorVarphi}, we get
  \begin{equation}
      a_0(\varphi^h,\varphi^h) = -a_0(\eta,\varphi^h)
      - a_1(\eta,\psi,\varphi^h) - a_1(\varphi^h,\psi,\varphi^h) 
      - a_1(\psi^h,\eta,\varphi^h) - a_2(\eta,\varphi^h).
    \label{eqn:ErrorZeroed}
  \end{equation}
  Using
    $a_0(\varphi^h,\varphi^h) = Re^{-1} \, |\varphi^h|^2_2$
  and \eqref{eqn:a0} -- \eqref{eqn:a2} in \eqref{eqn:ErrorZeroed}, simplifying, and rearranging terms, gives
  \begin{equation}
      |\varphi^h|_2 
      \le 
      \left(
      Re^{-1}
      - \Gamma_1 \, | \psi |_2
      \right)^{-1} \, 
      \left(
      Re^{-1} 
      + \Gamma_1 \, |\psi|_2 
      + \Gamma_1 \, |\psi^h|_2 
      + Ro^{-1} \, \Gamma_2
      \right) \,  
      |\eta|_2 .
    \label{eqn:phihIneq}
  \end{equation}
  Using \eqref{eqn:phihIneq} and the triangle inequality along with the stability estimates \eqref{eqn:stability_1} and \eqref{eqn:stability_fem_1}, gives:
  \begin{align}
    |e|_2 
    &\le |\eta|_2 
    + |\varphi^h|_2 
    \le \left[
    1 
    + \frac{Re^{-1} 
      + \Gamma_1 \, |\psi|_2 
      + \Gamma_1 \, |\psi^h|_2 
      + Ro^{-1} \, \Gamma_2} {Re^{-1} 
      - \Gamma_1 \, |\psi|_2} 
      \right] \, |\eta|_2 \nonumber \\[0.2cm]
  &= 
  \left[
  \frac{
  Ro^{-1} \, \Gamma_2
  + 2 \, Re^{-1} 
  + \Gamma_1 \, Re \, Ro^{-1} \, \| F \|_{-2}
  }
  {
  Re^{-1}
  - \Gamma_1 \, Re \, Ro^{-1} \, \| F \|_{-2}
  }
  \right] \, | \psi-\lambda^h |_2 ,
    \label{eqn:EnergyError}
  \end{align}
where $\lambda^h \in X^h$ is arbitrary.
Taking the infimum over $\lambda^h \in X^h$ in \eqref{eqn:EnergyError} proves estimate \eqref{eqn:EnergyNorm}.
\hfill 
\end{proof}

Next, we prove error estimates in the $L^2$ norm and $H^1$ seminorm by using a duality argument.
To this end, we first notice that the QGE \eqref{qge_psi_1} can be written as
\begin{eqnarray}
\mathcal{N} \, \psi
= Ro^{-1} \, F ,
\label{eqn:qge_operator_formulation}
\end{eqnarray}
where the nonlinear operator $\mathcal{N}$ is defined as
\begin{eqnarray}
\mathcal{N} \, \psi
:= Re^{-1} \, \Delta^2 \psi 
+ J(\psi , \Delta \psi)
- Ro^{-1} \, \frac{\partial \psi}{\partial x} .
\label{eqn:nonlinear_operator}
\end{eqnarray}
The linearization of $\mathcal{N}$ around $\psi$, a solution of \eqref{qge_psi_1}, yields the following {\it linear} operator:
\begin{eqnarray}
\mathcal{L} \, \chi
:= Re^{-1} \, \Delta^2 \chi 
+ J(\chi , \Delta \psi)
+ J(\psi, \Delta \chi)
- Ro^{-1} \, \frac{\partial \chi}{\partial x} .
\label{eqn:linear_operator}
\end{eqnarray}
To find the {\it dual operator} $\mathcal{L}^*$ of $\mathcal{L}$, we use \eqref{eqn:linear_operator} and apply Green's theorem:
\begin{eqnarray}
\hspace*{-0.4cm}
(\mathcal{L} \, \chi , \psi^*)
&=& \left( 
Re^{-1} \, \Delta^2 \chi 
+ J(\chi , \Delta \psi)
+ J(\psi, \Delta \chi)
- Ro^{-1} \, \frac{\partial \chi}{\partial x} 
\, , \, \psi^*
\right) 
\nonumber \\
&=& \left( 
\chi \, , \, 
Re^{-1} \, \Delta^2 \, \psi^*
- J(\psi , \Delta \psi^* )
+ Ro^{-1} \, \frac{\partial \psi^*}{\partial x} 
\right)
+ \biggl( \chi , J(\Delta \psi , \psi^*) \biggr)
= ( \chi , \mathcal{L}^* \, \psi^*) .
\label{eqn:dual_operator_1}
\end{eqnarray}
Thus, the dual operator $\mathcal{L}^*$ is given by
\begin{eqnarray}
\mathcal{L}^* \, \psi^*
= Re^{-1} \, \Delta^2 \, \psi^*
- J(\psi , \Delta \psi^* )
+ J(\Delta \psi , \psi^* )
+ Ro^{-1} \, \frac{\partial \psi^*}{\partial x} .
\label{eqn:dual_operator_4}
\end{eqnarray}
For any given $g \in L^2(\Omega)$, the weak formulation of the {\it dual problem} is:
\begin{eqnarray}
( \mathcal{L}^* \, \psi^* , \chi )
= (g , \chi)
\qquad
\forall \, \chi \in X = H_0^2(\Omega) .
\label{eqn:dual_operator_5}
\end{eqnarray}
We assume that $\psi^*$, the solution of \eqref{eqn:dual_operator_5}, satisfies the following elliptic regularity estimates:
\begin{eqnarray}
&& \psi^* \in H^4(\Omega) \cap H^2_0(\Omega), 
\label{eqn:dual_operator_6a} \\[0.2cm]
&& \| \psi^* \|_4 
\le C \, \|g\|_{0},
\label{eqn:dual_operator_6b} \\[0.2cm]
&& \| \psi^* \|_3 
\le C \, \|g\|_{-1},
\label{eqn:dual_operator_6c}
\end{eqnarray}
where $C$ is a generic constant that can depend on the data, but not on the meshsize $h$.
\begin{remark}
We note that this type of elliptic regularity was also assumed in \cite{Cayco86} for the streamfunction formulation of the 2D NSE.
In that report, it was also noted that, for a polygonal domain with maximum interior vertex angle $\theta < 126^{\circ}$, the assumed elliptic regularity was actually proved in \cite{blum1980boundary}.
We note that the theory developed in \cite{blum1980boundary} carries over to our case.
In Section 5 in \cite{blum1980boundary} it is proved that, for weakly nonlinear problems that involve the biharmonic operator as linear main part and that satisfy certain growth restrictions, each weak solution satisfies elliptic regularity results of the form \eqref{eqn:dual_operator_6a}-\eqref{eqn:dual_operator_6c}.
Assuming that $\Omega$ is a bounded polygonal domain with inner angle $\omega$ at each boundary corner satisfying $\omega < 126.283696\ldots^{\circ}$, Theorem 7 in \cite{blum1980boundary} with $k = 0$ and $k = 1$ implies \eqref{eqn:dual_operator_6a}-\eqref{eqn:dual_operator_6c}.
Using an argument similar to that used in Section 6(b) in \cite{blum1980boundary} to prove that the streamfunction formulation of the 2D NSE satisfies the restrictions in Theorem 7, we can prove that $\psi^*$, the solution of our dual problem \eqref{eqn:dual_operator_5}, satisfies the elliptic regularity results in \eqref{eqn:dual_operator_6a}-\eqref{eqn:dual_operator_6c}.
Indeed, the main point in Section 6(b) in \cite{blum1980boundary} is that the corner singularities arising in flows around sharp corners are essentially determined by the linear main part $\Delta^2$ in the streamfunction formulation of the 2D NSE, which is the linear main part of our dual problem \eqref{eqn:dual_operator_5} as well.
\end{remark}

\begin{theorem} \label{thm:Errors}
  Let $\psi$ be the solution of \eqref{eqn:WeakForm} and $\psi^h$ be the solution
  of \eqref{eqn:FEForm}. 
  Assume that the same small data condition as in Theorem \ref{thm:EnergyNorm} is satisfied:
  \begin{eqnarray}
  Re^{-2} \, Ro 
  \geq \Gamma_1 \, \| F \|_{-2} .
  \label{eqn:small_data_condition_dual}
  \end{eqnarray}
  Furthermore, assume that $\psi\in H^6(\Omega) \cap H^2_0(\Omega)$. 
  Then there exist positive constants $C_0, \, C_1 \text{ and } C_2$ that can depend on $Re$, $Ro$, $\Gamma_1$, $\Gamma_2$, $F$, but {\it not} on the meshsize $h$, such that
  \begin{align}
    |\psi - \psi^h|_2 &\le C_2 \, h^4 , \label{eqn:H2Error} \\
    |\psi - \psi^h|_1 &\le C_1 \, h^5 , \label{eqn:H1Error} \\
    \|\psi - \psi^h\|_0 &\le C_0 \, h^6 . \label{eqn:L2Error}
  \end{align}
\end{theorem}
\begin{remark}
The Argyris FE error estimates in Theorem~\ref{thm:Errors} can be extended to other conforming $C^1$ FE spaces. 
\end{remark}
\begin{proof}
Estimate \eqref{eqn:H2Error} follows immediately from \eqref{eqn:argyris_approximation_0} and Theorem \ref{thm:EnergyNorm}. 
Estimates \eqref{eqn:L2Error} and \eqref{eqn:H1Error} follow from a duality argument.

The error in the primal problem \eqref{eqn:WeakForm} and the interpolation error in the dual problem  \eqref{eqn:dual_operator_5} (with the function $g$ to be specified later) are denoted as $e := \psi - \psi^h$ and $e^* : = \psi^* - {\psi^*}^h$, respectively.

To prove the $L^2$ norm estimate \eqref{eqn:L2Error}, we consider $g = e$ in the dual problem \eqref{eqn:dual_operator_5}:
\begin{eqnarray}
|e|^2
= (e , e)
= (\mathcal{L} \, e , \psi^*)
= (e , \mathcal{L}^* \, \psi^*)
= (e , \mathcal{L}^* \, e^*)
+ (e , \mathcal{L}^* \, {\psi^*}^h)
= (\mathcal{L} \, e , e^*)
+ (\mathcal{L} \, e , {\psi^*}^h) .
\label{eqn:theorem_dual_2}
\end{eqnarray}
The last term on the right-hand side of \eqref{eqn:theorem_dual_2} is given by
\begin{eqnarray}
(\mathcal{L} \, e , {\psi^*}^h)
= \left(
Re^{-1} \, \Delta^2 e
+ J(e , \Delta \, \psi)
+ J(\psi , \Delta \, e)
- Ro^{-1} \, \frac{\partial e}{\partial x}
 \, , \,  
{\psi^*}^h
\right) .
\label{eqn:theorem_dual_3}
\end{eqnarray}
To estimate this term, we consider the error equation obtained by subtracting \eqref{eqn:FEForm} (with $\psi^h= {\psi^*}^h$) from \eqref{eqn:WeakForm} (with $\chi = {\psi^*}^h$):
\begin{eqnarray}
\left(
Re^{-1} \, \Delta^2 e
- Ro^{-1} \, \frac{\partial e}{\partial x}
 \, , \,  
{\psi^*}^h
\right)
+ \left(
J(\psi , \Delta \, \psi)
- J(\psi^h , \Delta \, \psi^h)
 \, , \,  
{\psi^*}^h
\right)
= 0 .
\label{eqn:theorem_dual_4}
\end{eqnarray}
Using \eqref{eqn:theorem_dual_4}, equation \eqref{eqn:theorem_dual_3} can be written as follows:
\begin{eqnarray}
(\mathcal{L} \, e , {\psi^*}^h)
= \left(
J(e , \Delta \, \psi)
+ J(\psi , \Delta \, e)
- J(\psi , \Delta \, \psi)
+ J(\psi^h , \Delta \, \psi^h)
 \, , \,  
{\psi^*}^h
\right) .
\label{eqn:theorem_dual_5}
\end{eqnarray}
Thus, by using \eqref{eqn:theorem_dual_5} equation \eqref{eqn:theorem_dual_2} becomes
\begin{eqnarray}
|e|^2
&=& (\mathcal{L} \, e , e^*)
+ (\mathcal{L} \, e , {\psi^*}^h)
\nonumber \\
&=& a_0(e , e^*)
+ a_2(e , e^*)
+ a_1(e , \psi , e^*)
+ a_1(\psi , e, e^*)
+ a_1(e , \psi , {\psi^*}^h)
\nonumber \\
&& + a_1(\psi , e , {\psi^*}^h)
- a_1(\psi , \psi , {\psi^*}^h)
+ a_1(\psi^h , \psi^h , {\psi^*}^h)
\nonumber \\
&=& a_0(e , e^*)
+ a_2(e , e^*)
+ a_1(e , \psi , e^*)
+ a_1(\psi , e, e^*)
\nonumber \\
&& - a_1(e , \psi , e^*)
+ a_1(e , \psi^h , e^*)
+ a_1(e , e , \psi^*)
\label{eqn:theorem_dual_6}
\end{eqnarray}
Using the bounds in \eqref{eqn:a0Cont}-\eqref{eqn:a2Cont}, \eqref{eqn:theorem_dual_6} yields
\begin{eqnarray}
\hspace*{-0.7cm} |e|^2
&\leq& Re^{-1} \, | e |_2 \, |e^* |_2
+ Ro^{-1} \, \Gamma_2 \, | e |_2 \, |e^* |_2
+ \Gamma_1 \, | e |_2 \, | \psi |_2 \, | e^* |_2
+ \Gamma_1 \, | \psi |_2 \, | e |_2 \, | e^* |_2
\nonumber \\
&& + \Gamma_1 \, | e |_2 \, | \psi |_2 \, | e^* |_2
+ \Gamma_1 \, | e |_2 \, | \psi^h |_2 \, | e^* |_2
+ \Gamma_1 \, | e |_2 \, | e |_2 \, | \psi^* |_2
\nonumber \\
&=& | e |_2 \, |e^* |_2 \, 
\left(
Re^{-1}
+ Ro^{-1} \, \Gamma_2
+ \Gamma_1 \, | \psi |_2 
+ \Gamma_1 \, | \psi |_2
+ \Gamma_1 \, | \psi |_2
+ \Gamma_1 \, | \psi^h |_2
\right)
+ | e |_2^2 \, 
\left(
\Gamma_1 \, | \psi^* |_2
\right). \quad
\label{eqn:theorem_dual_7}
\end{eqnarray}
Using the stability estimates \eqref{eqn:stability_1} and \eqref{eqn:stability_fem_1}, \eqref{eqn:theorem_dual_7} becomes
\begin{eqnarray}
|e|^2
\leq C \, | e |_2 \, |e^* |_2
+ | e |_2^2 \, 
\left(
\Gamma_1 \, | \psi^* |_2
\right) ,
\label{eqn:theorem_dual_8}
\end{eqnarray}
where $C$ is a generic constant that can depend on $Re$, $Ro$, $\Gamma_1$, $\Gamma_2$, $F$, but {\it not} on the meshsize $h$.
Using the approximation results \eqref{eqn:argyris_approximation_1}, we get
\begin{eqnarray}
|e^* |_2
\leq C \, h^2 \, | \psi^* |_4 .
\label{eqn:theorem_dual_9}
\end{eqnarray}
Using \eqref{eqn:dual_operator_6a}-\eqref{eqn:dual_operator_6b}, the elliptic regularity results of the dual problem \eqref{eqn:dual_operator_5} with $g := e$, we get
\begin{eqnarray}
| \psi^* |_4
\leq C \, | e | ,
\label{eqn:theorem_dual_10}
\end{eqnarray}
which obviously implies
\begin{eqnarray}
| \psi^* |_2
\leq C \, | e | .
\label{eqn:theorem_dual_11}
\end{eqnarray}
Inequalities \eqref{eqn:theorem_dual_9}-\eqref{eqn:theorem_dual_10} imply
\begin{eqnarray}
|e^* |_2
\leq C \, h^2 \, | e | .
\label{eqn:theorem_dual_12}
\end{eqnarray}
Inserting \eqref{eqn:theorem_dual_11} and \eqref{eqn:theorem_dual_12}  in \eqref{eqn:theorem_dual_8}, we get
\begin{eqnarray}
|e|^2
\leq C \, h^2 \, | e |_2 \, | e |
+ C \, | e |_2^2 \, | e | .
\label{eqn:theorem_dual_13}
\end{eqnarray}
Using the obvious simplifications and the $H^2$ error estimate \eqref{eqn:H2Error} in \eqref{eqn:theorem_dual_13} yields
\begin{eqnarray}
|e|
\leq C \, h^2 \, | e |_2
+ C \, | e |_2^2
\leq C \, h^6 + C \, h^8
= C_0 \, h^6 ,
\label{eqn:theorem_dual_14}
\end{eqnarray}
which proves the $L^2$ error estimate \eqref{eqn:L2Error}.

Estimate \eqref{eqn:H1Error} can be proven using the same duality argument as that used to prove estimate \eqref{eqn:L2Error}.
The major differences are that we use $g = - \Delta e$ in the dual problem \eqref{eqn:dual_operator_5} and we use the approximation result \eqref{eqn:argyris_approximation_2}. 
\hfill
\end{proof}
  \section{Numerical Results} \label{sec:NumericalResults}
The main goal of this section is twofold. 
First, we show that the FE discretization of the streamfunction formulation of the QGE \eqref{eqn:FEForm} with the Argyris element produces accurate numerical approximations, which are close to those in the published literature \cite{Cascon,Myers,Vallis06}. 
Second, we show that the numerical results follow the theoretical error estimates in \autoref{thm:EnergyNorm} and
\autoref{thm:Errors}.

\subsection{Mathematical Models}
\label{sec:mathematical_models}

Although the pure streamfunction formulation of the steady QGE \eqref{qge_psi_1} is our main concern,
we also test our Argyris FE discretization on two simplified settings: \begin{inparaenum}[(i)] \item the
\emph{Linear Stommel} model; and \item the \emph{Linear Stommel-Munk} model. \end{inparaenum} The reason
for using these two additional numerical tests is that they are standard test problems in the
geophysical fluid dynamics literature (see, e.g., Chapter 14 in Vallis \cite{Vallis06} as well as
the reports of Myers and Weaver \cite{Myers} and Cascon \emph{et al.} \cite{Cascon}). This
allows us to benchmark our numerical results against those in the published literature. Since both
the Linear Stommel and the Linear Stommel-Munk models lack the nonlinearity present in the QGE
\eqref{qge_psi_1}, they represent good stepping stones for testing our FE discretization.

The \emph{Linear Stommel-Munk} model (see equation (14.42) in \cite{Vallis06} and Problem 2 in \cite{Cascon}) is
\begin{equation}
  \epsilon_S \Delta \psi - \epsilon_M \Delta^2 \psi + \frac{\partial \psi}{\partial x} = f.
  \label{eqn:Stommel-Munk}
\end{equation}
The parameters $\epsilon_S \text{ and } \epsilon_M$ in \eqref{eqn:Stommel-Munk} are the {\it Stommel number} and {\it Munk scale}, respectively, which are given by (see, e.g., equation (10) in \cite{Myers} and equations (14.22) and (14.44) in \cite{Vallis06}) 
$\epsilon_M = \frac{A}{\beta L^3}$ and $\epsilon_S = \frac{\gamma}{\beta L}$, 
where
$A$ is the eddy viscosity parameterization,
$\beta$ is the coefficient multiplying the $y$ coordinate in the $\beta$-plane approximation,
$L$ is the width of the computational domain, and
$\gamma$ is the coefficient of the linear drag (Rayleigh friction) as might be generated by a bottom Ekman layer (see equation (14.5) in \cite{Vallis06}).
The model is
supplemented with appropriate boundary conditions, which will be described for each of the
subsequent numerical tests.

We note that the Linear Stommel-Munk model \eqref{eqn:Stommel-Munk} is similar in form to the QGE
\eqref{qge_psi_1}. Indeed, both models contain the biharmonic operator $\Delta^2 \psi$, the rotation
term $\frac{\partial \psi}{\partial x}$, and the forcing term $f$. The two main differences between
the two models are the following: First, the QGE are nonlinear, since they contain the Jacobian term
$J(\psi,q)$, whereas the Stommel-Munk model is linear. The
second difference is that the Linear Stommel-Munk model contains a Laplacian term $\Delta \psi$,
whereas the QGE do not.

We also note that the two models use different parameters: the Reynolds number, $Re$, and the Rossby
number, $Ro$, in the QGE and the Stommel number, $\epsilon_S$, and the Munk scale, $\epsilon_M$,
in the Linear Stommel-Munk model. 
The parameters $\epsilon_M$, $Ro$, and $Re$ are related through $\epsilon_M = Ro\, Re^{-1}$.
%
There is, however, no explicit relationship among $\epsilon_S$, $Ro$, and $Re$.
The reason is that the QGE \eqref{qge_psi_1} do not contain the Laplacian term that is present in the Stommel-Munk model \eqref{eqn:Stommel-Munk}, which models the bottom Rayleigh friction.
Thus, the coefficient $\gamma$ does not have a counterpart in the QGE.
This explains why $\epsilon_S$, which depends on $\gamma$, cannot be directly expressed as a function of $Ro$ and $Re$.

The second simplified model used in our numerical investigation is the \emph{Linear Stommel} model
(see, e.g., equation (14.22) in \cite{Vallis06} and equation (11) in \cite{Myers}):
\begin{equation}
  \epsilon_S \Delta \psi + \frac{\partial \psi}{\partial x} = f.
  \label{eqn:Stommel}
\end{equation}
We note that the Linear Stommel model \eqref{eqn:Stommel} is just the Linear Stommel-Munk model
\eqref{eqn:Stommel-Munk} in which the biharmonic term is dropped (i.e., $\epsilon_M=0$).

\subsection{Numerical Tests} \label{sec:Tests}

In this section, we present results for the Linear Stommel model \eqref{eqn:Stommel}, the Linear Stommel-Munk model \eqref{eqn:Stommel-Munk}, and the (nonlinear) QGE \eqref{qge_psi_1}. 

\subsubsection{Linear Stommel Model} \label{sss:LSM}
This section presents the results for the FE discretization of the Linear Stommel model
\eqref{eqn:Stommel} by using the Argyris element. The computational domain is $\Omega = [0,1]\times
[0,1]$. For completeness, we present results for two numerical tests. The first test, denoted by
Test 1, corresponds to the exact solution used by Vallis (equation (14.38) in \cite{Vallis06}), while
the second test, denoted by Test 2, corresponds to the exact solution used by Myers and Weaver (equations (15) and (16) in \cite{Myers}).

\tbf{Test 1a:}
In this test, we choose the same setting as that used in equation (14.38) in \cite{Vallis06}.
In particular, the forcing term and the non-homogeneous Dirichlet boundary conditions are chosen to
match those given by the exact solution
  $\psi(x,y) = (1-x-e^{-\nicefrac{x}{\epsilon_S}}) \sin \left( \pi y \right)$.
We choose the same Stommel number as that used in \cite{Vallis06}, i.e., $\epsilon_S=0.04$. 

Figure \ref{fig:StommelVallis} presents the streamlines of the approximate solution obtained by using
the Argyris element on a mesh with $h=\frac{1}{32}$ and $9670$ DoFs. 
We note that Figure \ref{fig:StommelVallis} resembles Figure $14.5$ in \cite{Vallis06}.
Since the exact solution is available, we can compute the errors in various norms.
\autoref{tab:StommelErrorsVallis} presents the errors $e_0,\, e_1, \text{ and } e_2$ (i.e., the
$L^2,\, H^1, \text{ and } H^2$ errors, respectively) for various values of the meshsize, $h$ (the
DoFs are also included).
We note that the errors in \autoref{tab:StommelErrorsVallis} follow the theoretical rates of
convergence predicted by the estimates \eqref{eqn:H2Error}--\eqref{eqn:L2Error} in
\autoref{thm:Errors}. The orders of convergence in \autoref{tab:StommelErrorsVallis} are close to
the theoretical ones for the fine meshes, but not as close for the coarse meshes. We think that the
inaccuracies on the coarse meshes are due to their inability to capture the thin boundary layer at $x=0$. 
The finer the mesh gets, the better this boundary layer is captured and the better the numerical accuracy becomes.

\begin{table}
\begin{center}
{\footnotesize
\begin{tabular}{|c|c|c|c|c|c|c|c|}
  \hline
  $h$ & $DoFs$ & $e_0$ & $L_2$ order & $e_1$ & $H^1$ order & $e_2$ & $H^2$ order \\ 
  \hline
  $\nicefrac{1}{2}$ & $70$ & $0.1148$ & $-$ & $1.81$ & $-$ & $83.67$ & $-$ \\
  $\nicefrac{1}{4}$ & $206$ & $0.01018$ & $3.495$ & $0.312$ & $2.537$ & $25.48$ & $1.716$ \\
  $\nicefrac{1}{8}$ & $694$ & $0.0004461$ & $4.512$ & $0.02585$ & $3.593$ & $3.902$ & $2.707$ \\
  $\nicefrac{1}{16}$ & $2534$ & $1.09\times 19^{-5}$ & $5.355$ & $0.001215$ & $4.412$ & $0.3494$ & $3.481$ \\
  $\nicefrac{1}{32}$ & $9670$ & $1.972\times 19^{-7}$ & $5.788$ & $4.349\times 19^{-5}$ & $4.804$ & $0.02335$ & $3.903$ \\
  \hline
\end{tabular}
}
\end{center}
\caption{
Linear Stommel Model \eqref{eqn:Stommel}, Test 1a \cite{Vallis06}: 
The errors $e_0,\, e_1,\, e_2$ for various meshsizes $h$. 
} 
\label{tab:StommelErrorsVallis}
\end{table}

\begin{figure}
  \begin{center}
    \subfigure[Test 1a \cite{Vallis06}.] {
      \includegraphics[trim=80 200 70 220, clip=true, scale=0.38]{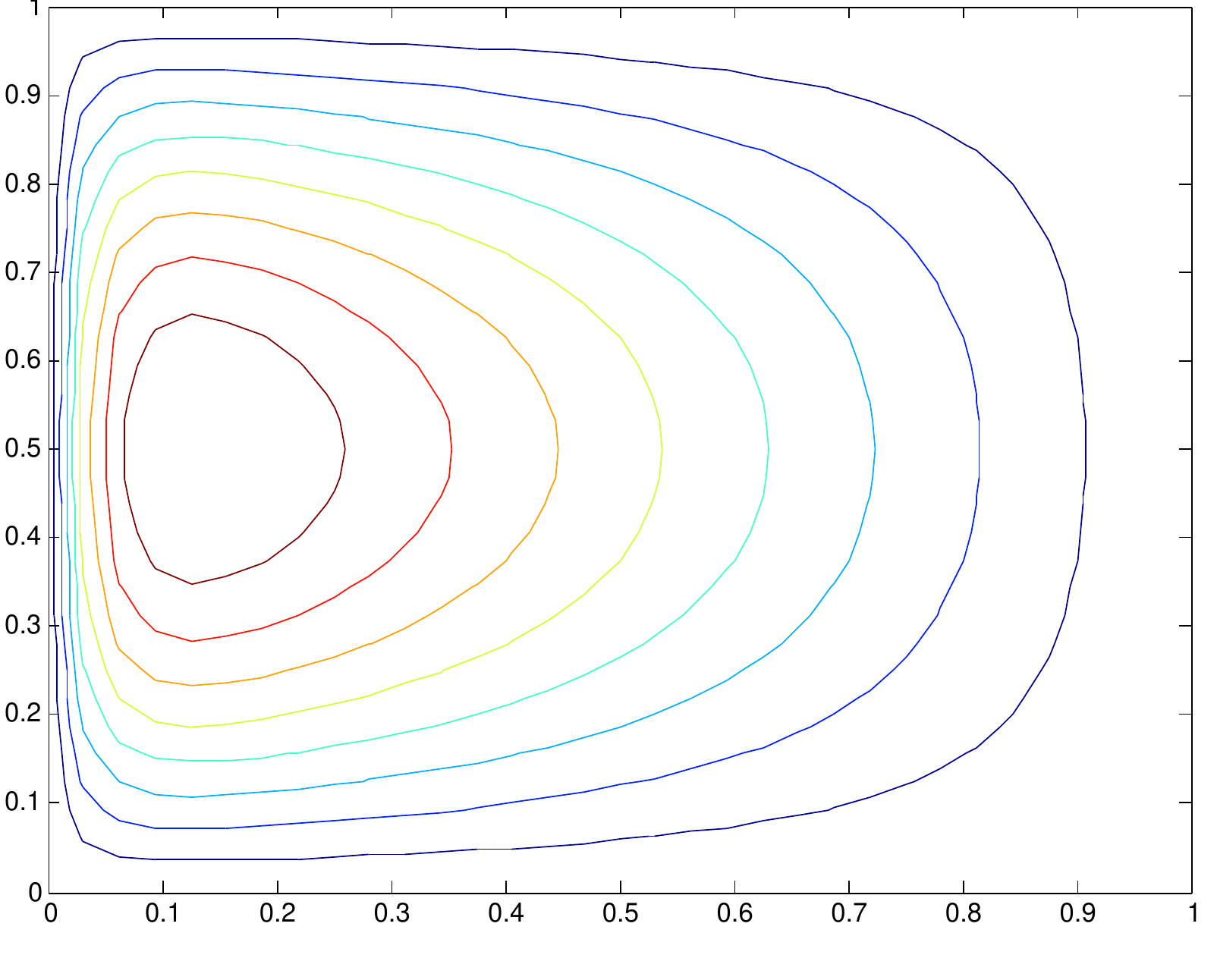}
      \label{fig:StommelVallis}
    }
    \subfigure[Test 1b \cite{Vallis06}.]{
      \includegraphics[trim=80 200 70 220, clip=true, scale=0.38]{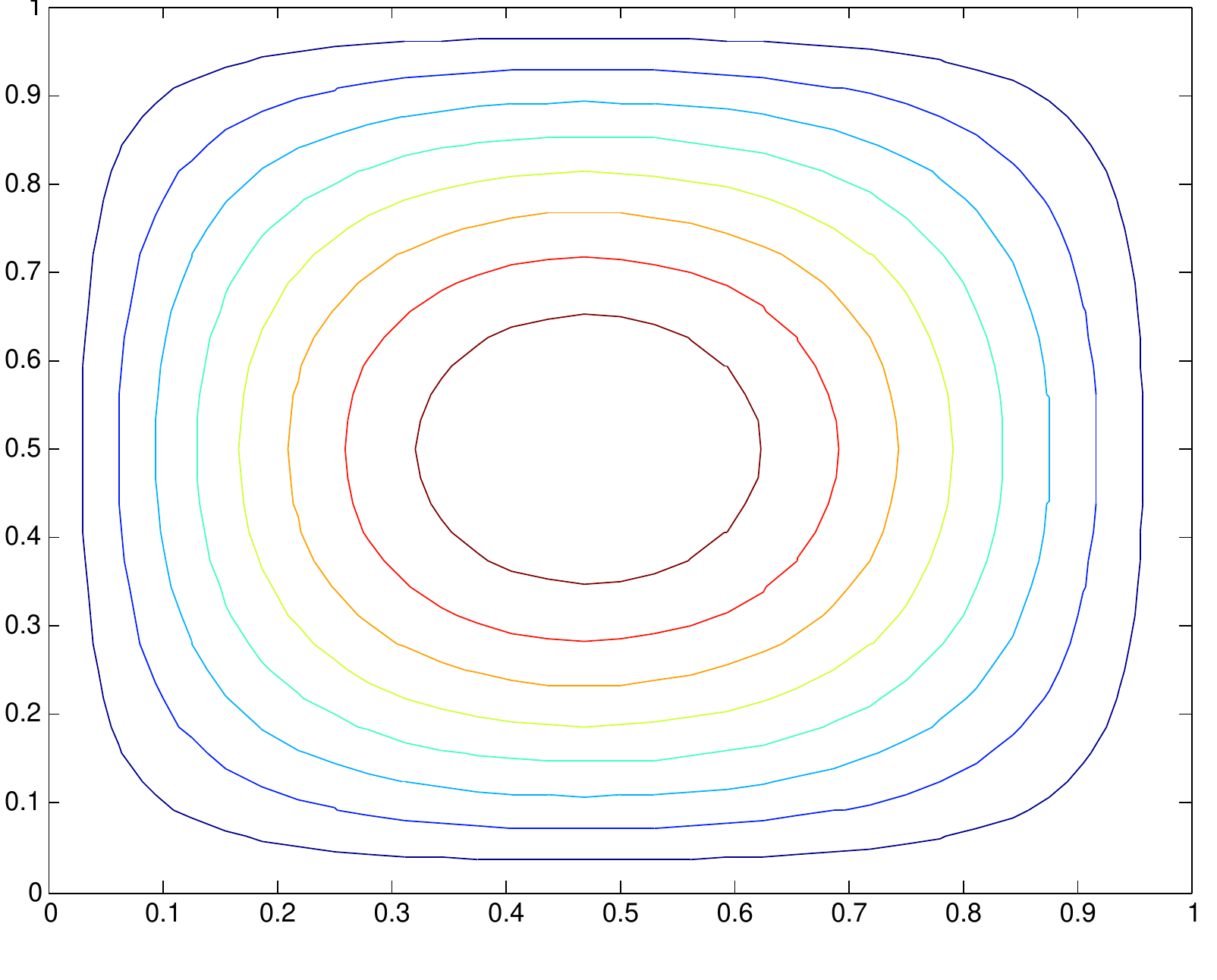}
      \label{fig:StommelVallise1}
    }
    \caption{Linear Stommel Model \eqref{eqn:Stommel}: Streamlines of the approximation, $\psi^h$,
    on a mesh with $h = \frac{1}{32}$. 
    }
    \label{fig:Stommel1}
  \end{center}
\end{figure}

\tbf{Test 1b:}
To verify whether the degrading accuracy of the approximation is indeed due to the thin (western) boundary layer, we use $\epsilon_S=1$ in Test 1a, which will result in a much thicker western boundary layer. 
We then run Test 1a, but with the new $\epsilon_S$. 
As can be seen in \autoref{tab:StommelErrorsVallise1}, the rates of convergence are the expected theoretical orders of convergence. 
This shows that the reason for the inaccuracies in \autoref{tab:StommelErrorsVallis} were indeed due to the thin western boundary layer.

\begin{table}
\begin{center}
{\scriptsize
\begin{tabular}{|c|c|c|c|c|c|c|c|}
  \hline
  $h$ & $DoFs$ & $e_0$ & $L_2$ order & $e_1$ & $H^1$ order & $e_2$ & $H^2$ order \\ 
  \hline
  $\nicefrac{1}{2}$ & $70$ & $1.689\times 10^{-5}$ & $-$ & $0.0003434$ & $-$ & $0.008721$ & $-$ \\
  $\nicefrac{1}{4}$ & $206$ & $3.722\times 10^{-7}$ & $5.504$ & $1.341\times 10^{-5}$ & $4.678$ & $0.0005616$ & $3.957$ \\
  $\nicefrac{1}{8}$ & $694$ & $4.891\times 10^{-9}$ & $6.25$ & $3.757\times 10^{-7}$ & $5.158$ & $3.25\times 10^{-5}$ & $4.111$ \\
  $\nicefrac{1}{16}$ & $2534$ & $7.079\times 10^{-11}$ & $6.111$ & $1.117\times 10^{-8}$ & $5.071$ & $1.964\times 10^{-6}$ & $4.049$ \\
  $\nicefrac{1}{32}$ & $9670$ & $1.08\times 10^{-12}$ & $6.035$ & $3.437\times 10^{-10}$ & $5.023$ & $1.213\times 10^{-7}$ & $4.018$ \\
  \hline
\end{tabular}
}
\end{center}
\caption{
Linear Stommel Model \eqref{eqn:Stommel}, Test 1b \cite{Vallis06}: 
The errors $e_0,\, e_1,\, e_2$ for various meshsizes $h$. 
} 
\label{tab:StommelErrorsVallise1}
\end{table}
 
\tbf{Test 2:}
For this test, we use the exact solution given by equations (15) and (16) in \cite{Myers}, i.e.,
$
  \psi(x,y) =\frac{\sin(\pi y)}{\pi(1+4\pi^2\epsilon_S^2)}\left\{2\pi\epsilon_S\sin(\pi x)+\cos(\pi x)+\frac{1}{e^{R_1}-e^{R_2}}\left[(1+e^{R_2})e^{R_1x}-(1+e^{R_1})e^{R_2x}\right]\right\},
$
where 
  $R_{1,2} = \frac{-1\pm\sqrt{1+4\pi^2 \epsilon_S^2}}{2\epsilon_S}$.
The forcing and the homogeneous Dirichlet boundary conditions are chosen to match those given
by the exact solution. 
We choose the same Stommel number as that used in \cite{Myers}, i.e., $\epsilon_S=0.05$.

\autoref{fig:StommelMyers} presents the streamlines of the approximate solution obtained by using the Argyris element on a mesh with $h=\frac{1}{32}$ and $9670$ DoFs. 
We note that Figure \ref{fig:StommelMyers} resembles Figure $2$ in \cite{Myers}.
\autoref{tab:StommelErrorsMyers} presents the errors $e_0,\, e_1, \text{ and } e_2$ for various meshsizes $h$.
The errors in \autoref{tab:StommelErrorsMyers} follow the theoretical rates of convergence predicted by the
estimates \eqref{eqn:H2Error} - \eqref{eqn:L2Error} in \autoref{thm:Errors}. Again, we see that the orders of
convergence in \autoref{tab:StommelErrorsMyers} are close to the theoretical ones for the fine meshes, but not as close
for the coarse meshes. 
We again attribute this to the inaccuracies at the thin (western) boundary layer at $x=0$. 

\begin{table}
\begin{center}
{\footnotesize
\begin{tabular}{|c|c|c|c|c|c|c|c|}
  \hline
  $h$ & $DoFs$ & $e_0$ & $L_2$ order & $e_1$ & $H^1$ order & $e_2$ & $H^2$ order \\ 
  \hline
  $\nicefrac{1}{2}$ & $70$ & $0.005645$ & $-$ & $0.1451$ & $-$ & $6.602$ & $-$ \\
  $\nicefrac{1}{4}$ & $206$ & $0.0004276$ & $3.723$ & $0.02081$ & $2.801$ & $1.632$ & $2.016$ \\
  $\nicefrac{1}{8}$ & $694$ & $1.46\times 10^{-5}$ & $4.872$ & $0.001408$ & $3.886$ & $0.2066$ & $2.982$ \\
  $\nicefrac{1}{16}$ & $2534$ & $2.954\times 10^{-7}$ & $5.627$ & $5.829\times 10^{-5}$ & $4.594$ & $0.0165$ & $3.646$ \\
  $\nicefrac{1}{32}$ & $9670$ & $4.968\times 10^{-9}$ & $5.894$ & $1.998\times 10^{-6}$ & $4.867$ & $0.001069$ & $3.948$ \\
  \hline
\end{tabular}
}
\end{center}
\caption{Linear Stommel Model \eqref{eqn:Stommel}, Test 2 \cite{Myers}: 
The errors $e_0,\, e_1,\, e_2$ for various meshsizes $h$. 
} 
\label{tab:StommelErrorsMyers}
\end{table}

\begin{figure}
  \begin{center}
    \includegraphics[trim=80 200 70 220, clip=true, scale=0.38]{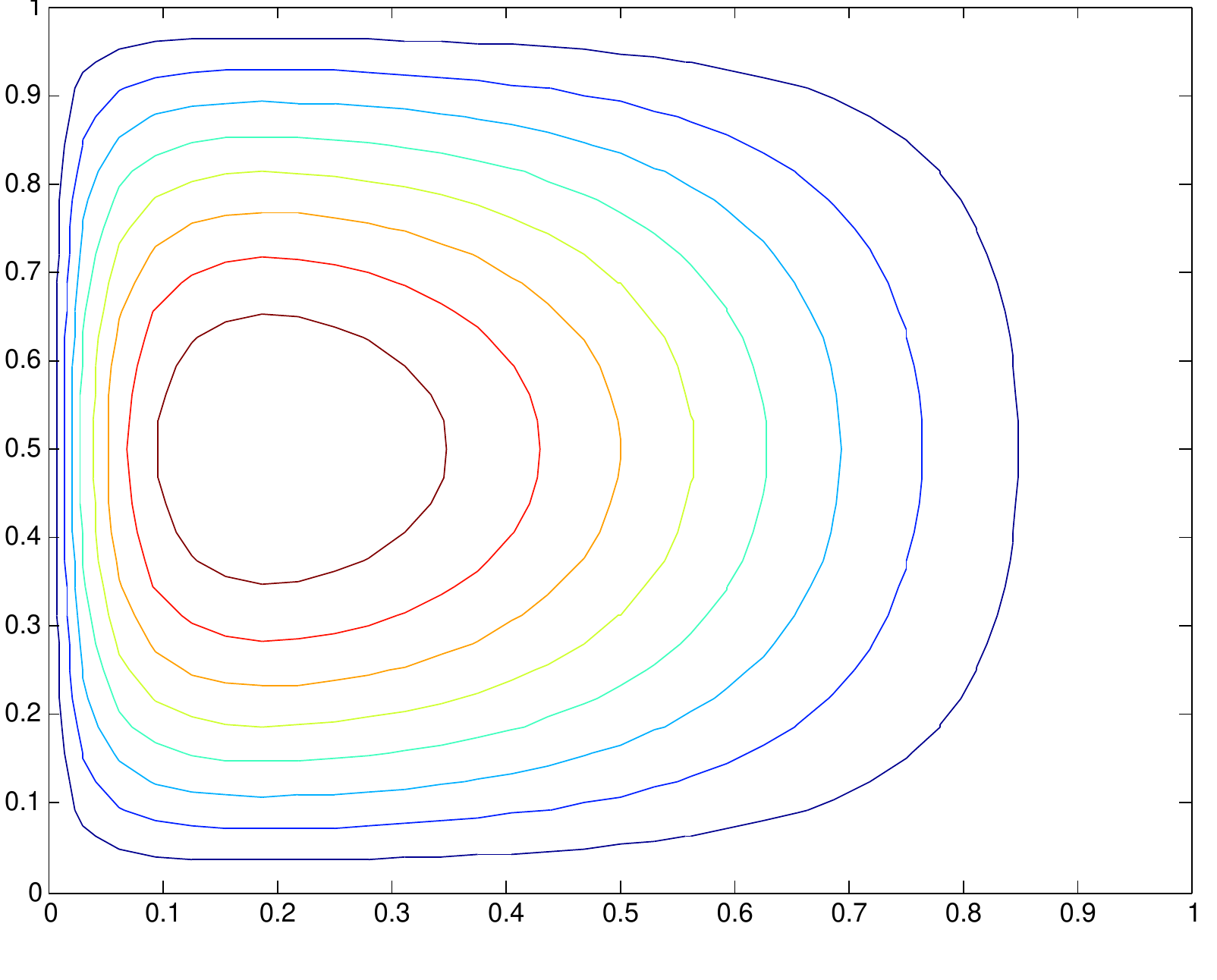}
    \caption{Linear Stommel Model \eqref{eqn:Stommel}, Test 2 \cite{Myers}: Streamlines of the
    approximation, $\psi^h$, on a mesh with $h=\frac{1}{32}$ and $9670$ DoFs. 
    }
    \label{fig:StommelMyers}
  \end{center}
\end{figure}

\subsubsection{Linear Stommel-Munk Model}\label{sss:SMM}
This section presents results for the FE discretization of the Linear Stommel-Munk model
\eqref{eqn:Stommel-Munk} by using the Argyris element. Our computational setting is the same as that
used by Cascon \emph{et al.} \cite{Cascon}: The computational domain is $\Omega = [0,3]\times[0,1]$,
the Munk scale is $\epsilon_M=6\times 10^{-5}$, the Stommel number is $\epsilon_S=0.05$, and the
boundary conditions are 
  $\psi = \frac{\partial \psi}{\partial \mathbf{n}}=0$ on $\partial\Omega$.
For completeness, we present results for two numerical tests, denoted by Test 3 and Test 4, corresponding to Test 1 and Test 2 in \cite{Cascon}, respectively. 

\tbf{Test 3:}
For this test, we use the exact solution given by Test 1 in \cite{Cascon}, i.e.,
  $\psi(x,y) = \sin^2 \left( \frac{\pi x}{3} \right) \sin^2 \left( \pi y \right)$.
The forcing term is chosen to match that given by the exact solution.


Figure \ref{fig:SMsin} presents the streamlines of the approximate solution obtained by using the Argyris element on a mesh with $h=\frac{1}{32}$ and $28550$ DoFs. 
We note that Figure \ref{fig:SMsin} resembles Figure $7$ in \cite{Myers}.
\autoref{tab:SMsinErrors} presents the errors $e_0,\, e_1, \text{ and } e_2$ for various meshsizes $h$.
The errors in \autoref{tab:SMsinErrors} follow the theoretical rates of convergence
predicted by the estimates \eqref{eqn:H2Error}--\eqref{eqn:L2Error} in \autoref{thm:Errors}.  This
time, we see that the orders of convergence in \autoref{tab:SMsinErrors} are close to the
theoretical ones for the fine meshes, but are higher than expected for the coarse meshes. We
attribute this to the fact that the exact solution does not display any
boundary layers that could be challenging to capture by the Argyris element on a coarse mesh.

\begin{table}
\begin{center}
{\scriptsize
\begin{tabular}{|c|c|c|c|c|c|c|c|}
  \hline
  $h$ & $DoFs$ & $e_0$ & $L_2$ order & $e_1$ & $H^1$ order & $e_2$ & $H^2$ order \\ 
  \hline
  $\nicefrac{1}{2}$ & $170$ & $0.00299$ & $-$ & $0.04084$ & $-$ & $0.7624$ & $-$ \\
  $\nicefrac{1}{4}$ & $550$ & $3.217\times 10^{-5}$ & $6.539$ & $0.001031$ & $5.308$ & $0.04078$ & $4.225$ \\
  $\nicefrac{1}{8}$ & $1958$ & $3.437\times 10^{-7}$ & $6.548$ & $2.491\times 10^{-5}$ & $5.371$ & $0.002253$ & $4.178$ \\
  $\nicefrac{1}{16}$ & $7366$ & $4.571\times 10^{-9}$ & $6.232$ & $7.026\times 10^{-7}$ & $5.148$ & $0.0001344$ & $4.067$ \\
  $\nicefrac{1}{32}$ & $28550$ & $6.704\times 10^{-11}$ & $6.091$ & $2.113\times 10^{-8}$ & $5.056$ & $8.26\times 10^{-6}$ & $4.024$ \\
 \hline
\end{tabular}
}
\end{center}
\caption{Linear Stommel-Munk Model \eqref{eqn:Stommel-Munk}, Test 3 \cite{Cascon}: The errors $e_0,\, e_1,\, e_2$  for various  meshsizes $h$. 
} 
\label{tab:SMsinErrors}
\end{table}

\begin{figure}
  \begin{center}
    \subfigure[Test 3 \cite{Cascon}.]{
      \includegraphics[trim=80 200 70 220, clip=true, scale=0.38]{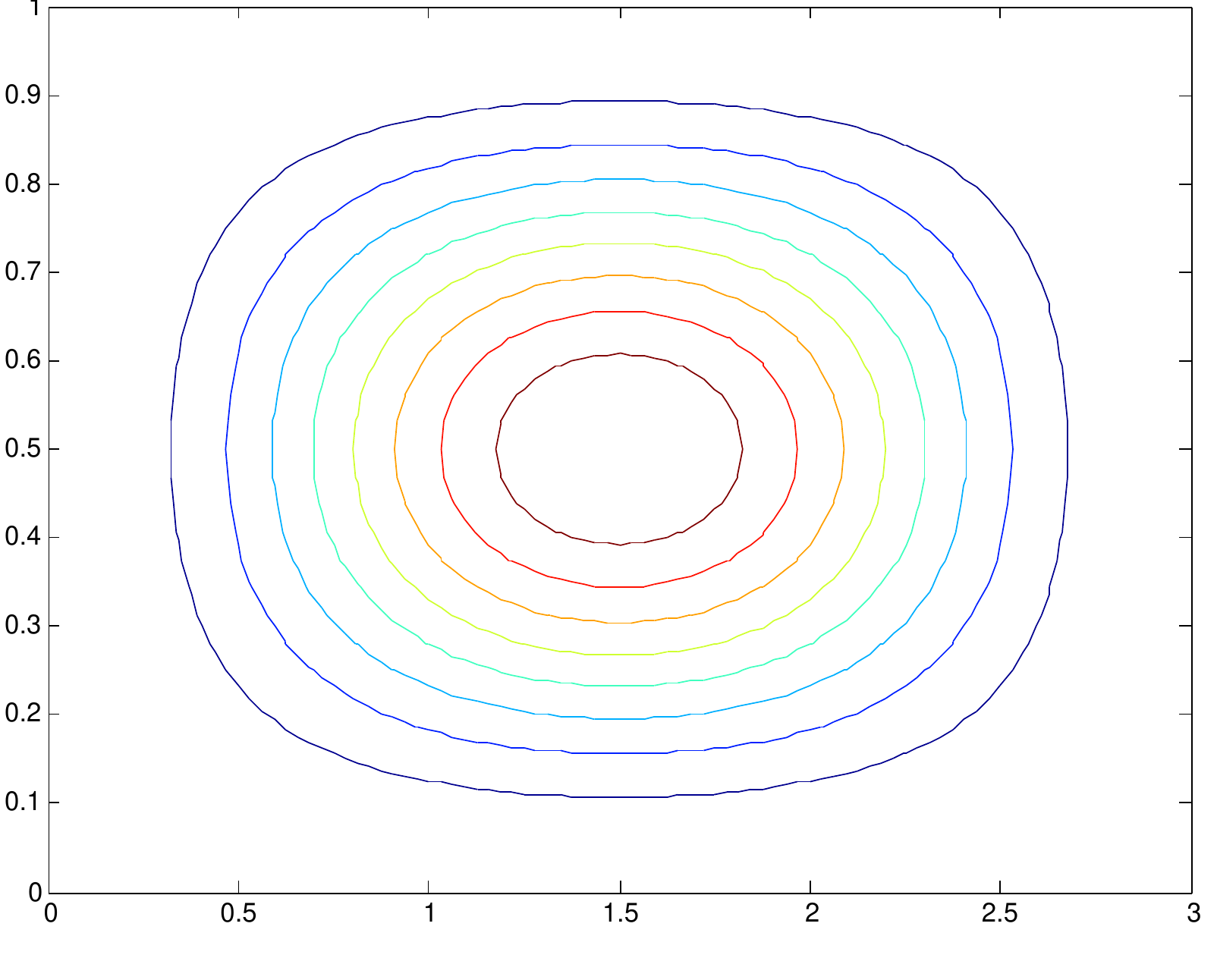}
      \label{fig:SMsin}
    }
    \subfigure[Test 4 \cite{Cascon}.]{ 
      \includegraphics[trim=80 200 70 220, clip=true, scale=0.38]{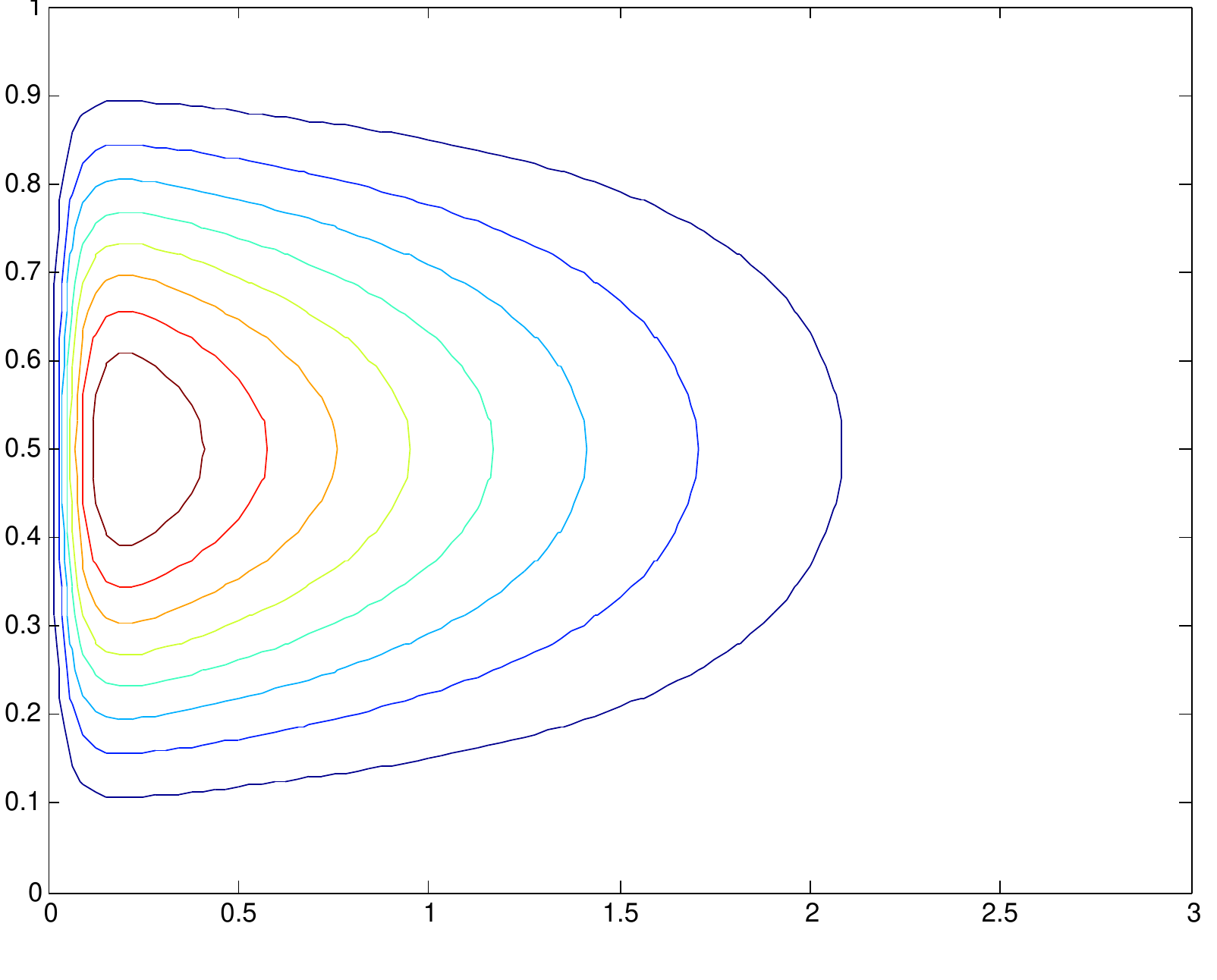}
      \label{fig:SMe}
    }
    \caption{Linear Stommel-Munk Model \eqref{eqn:Stommel-Munk}: Streamlines of the approximation,
      $\psi^h$, on a mesh with $h=\frac{1}{32}$ and $28550$ DoFs.  
      }
    \label{fig:StommelMunk}
  \end{center}
\end{figure}

\tbf{Test 4:}
For this test, we use the exact solution given by Test 2 in \cite{Cascon}, i.e.,
 $\psi(x,y) = \left[\left(1 - \frac{x}{3}\right)\left(1-e^{-20x}\right) \sin \left( \pi y \right) \right]^2$.
We take the forcing term $f$ corresponding to the exact solution. 

Figure \ref{fig:SMe} presents the streamlines of the approximate solution obtained by using the Argyris
element on a mesh with $h=\frac{1}{32}$ and $28550$ DoFs. 
We note that Figure \ref{fig:SMe} resembles Figure $10$ in \cite{Myers}.
\autoref{tab:SMeErrors} presents the errors $e_0,\, e_1, \text{ and } e_2$ for various meshsizes $h$.
We note that the errors in \autoref{tab:SMeErrors} follow the theoretical rates of convergence
predicted by the estimates \eqref{eqn:H2Error}--\eqref{eqn:L2Error} in \autoref{thm:Errors}. Again,
we see that the orders of convergence in \autoref{tab:SMeErrors} are close to the theoretical ones
for the fine meshes, but not as close for the coarse meshes. As stated previously, we attribute this
to the inaccuracies at the thin (western) boundary layer at $x=0$. 

\begin{table}
\begin{center}
{\small
\begin{tabular}{|c|c|c|c|c|c|c|c|}
  \hline
  $h$ & $DoFs$ & $e_0$ & $L_2$ order & $e_1$ & $H^1$ order & $e_2$ & $H^2$ order \\ 
  \hline
  $\nicefrac{1}{2}$ & $170$ & $0.06036$ & $-$ & $1.162$ & $-$ & $38.99$ & $-$ \\
  $\nicefrac{1}{4}$ & $550$ & $0.01132$ & $2.414$ & $0.3995$ & $1.541$ & $21.4$ & $0.8656$ \\
  $\nicefrac{1}{8}$ & $1958$ & $0.0008399$ & $3.753$ & $0.05914$ & $2.756$ & $5.656$ & $1.92$ \\
  $\nicefrac{1}{16}$ & $7366$ & $2.817\times 10^{-5}$ & $4.898$ & $0.004008$ & $3.883$ & $0.7378$ & $2.939$ \\
  $\nicefrac{1}{32}$ & $28550$ & $5.587\times 10^{-7}$ & $5.656$ & $0.0001607$ & $4.641$ & $0.0597$ & $3.627$ \\
 \hline
\end{tabular}
}
\end{center}
\caption{Linear Stommel-Munk Model \eqref{eqn:Stommel-Munk}, Test 4 \cite{Cascon}: The errors $e_0,\, e_1,\, e_2$ for various meshsizes $h$.} 
\label{tab:SMeErrors}
\end{table}

\subsubsection{Quasi-Geostrophic Equations}\label{sss:SQGE}
This section presents results for the FE discretization of the streamfunction formulation of the
QGE \eqref{qge_psi_1} by using the Argyris element. 
To solve the resulting nonlinear system of equations, we use Newton's method with the following stopping criteria:
the maximum residual norm is $10^{-8}$,
the maximum streamfunction iteration increment is $10^{-8}$, and 
the maximum number of iterations is 10. 
Our computational domain is
$\Omega=[0,3]\times[0,1]$, the Reynolds number is $Re=1.667$, and the Rossby number is $Ro=10^{-4}$.
For completeness, we present results for two numerical tests, denoted by Test 5 and Test 6, 
corresponding to the exact solutions given in Test 1 and Test 2 of \cite{Cascon}, respectively.

\tbf{Test 5:}
In this test, we take the same exact solution as that in Test 1 of \cite{Cascon}, i.e.,
  $\psi(x,y) = \sin^2 \left( \frac{\pi x}{3} \right) \sin^2 \left( \pi y \right)$.
The forcing term and homogeneous boundary conditions correspond to the exact solution. 

Figure \ref{fig:SQGEsin} presents the streamlines of the approximate solution obtained by using the
Argyris element on a mesh with $h=\frac{1}{32}$ and $28550$ DoFs. 
We note that Figure \ref{fig:SQGEsin} resembles Figure $7$ in \cite{Myers}.
\autoref{tab:SQGEsinErrors} presents the errors $e_0,\, e_1, \text{ and } e_2$ for various meshsizes $h$.
The errors in \autoref{tab:SQGEsinErrors} follow the theoretical rates of convergence
predicted by the estimates \eqref{eqn:H2Error}--\eqref{eqn:L2Error} in \autoref{thm:Errors}. Again,
since the exact solution does not display any boundary layers, we
see that the orders of convergence in \autoref{tab:SQGEsinErrors} are close to the theoretical ones for the fine meshes, but are higher than expected for the coarse meshes. 

\begin{table}
\begin{center}
{\scriptsize
\begin{tabular}{|c|c|c|c|c|c|c|c|}
  \hline
  $h$ & $DoFs$ & $e_0$ & $L_2$ order & $e_1$ & $H^1$ order & $e_2$ & $H^2$ order \\ 
  \hline
  $\nicefrac{1}{2}$ & $170$ & $0.005709$ & $-$ & $0.06033$ & $-$ & $1.087$ & $-$ \\
  $\nicefrac{1}{4}$ & $550$ & $3.726\times 10^{-5}$ & $7.259$ & $0.001086$ & $5.796$ & $0.04113$ & $4.724$ \\
  $\nicefrac{1}{8}$ & $1958$ & $3.597\times 10^{-7}$ & $6.695$ & $2.534\times 10^{-5}$ & $5.421$ & $0.002252$ & $4.191$ \\
  $\nicefrac{1}{16}$ & $7366$ & $4.648\times 10^{-9}$ & $6.274$ & $7.065\times 10^{-7}$ & $5.165$ & $0.0001344$ & $4.067$\\
  $\nicefrac{1}{32}$ & $28550$ & $6.737\times 10^{-11}$ & $6.108$ & $2.116\times 10^{-8}$ & $5.061$ & $8.26\times 10^{-6}$ & $4.024$ \\
 \hline
\end{tabular}
}
\end{center}
\caption{QGE \eqref{qge_psi_1}, Test 5: 
The errors $e_0,\, e_1,\, e_2$ for various meshsizes $h$. 
} 
\label{tab:SQGEsinErrors}
\end{table}

\begin{figure}
  \begin{center}
    \subfigure[Test 5.]{
      \includegraphics[trim=80 200 70 215, clip=true, scale=0.38]{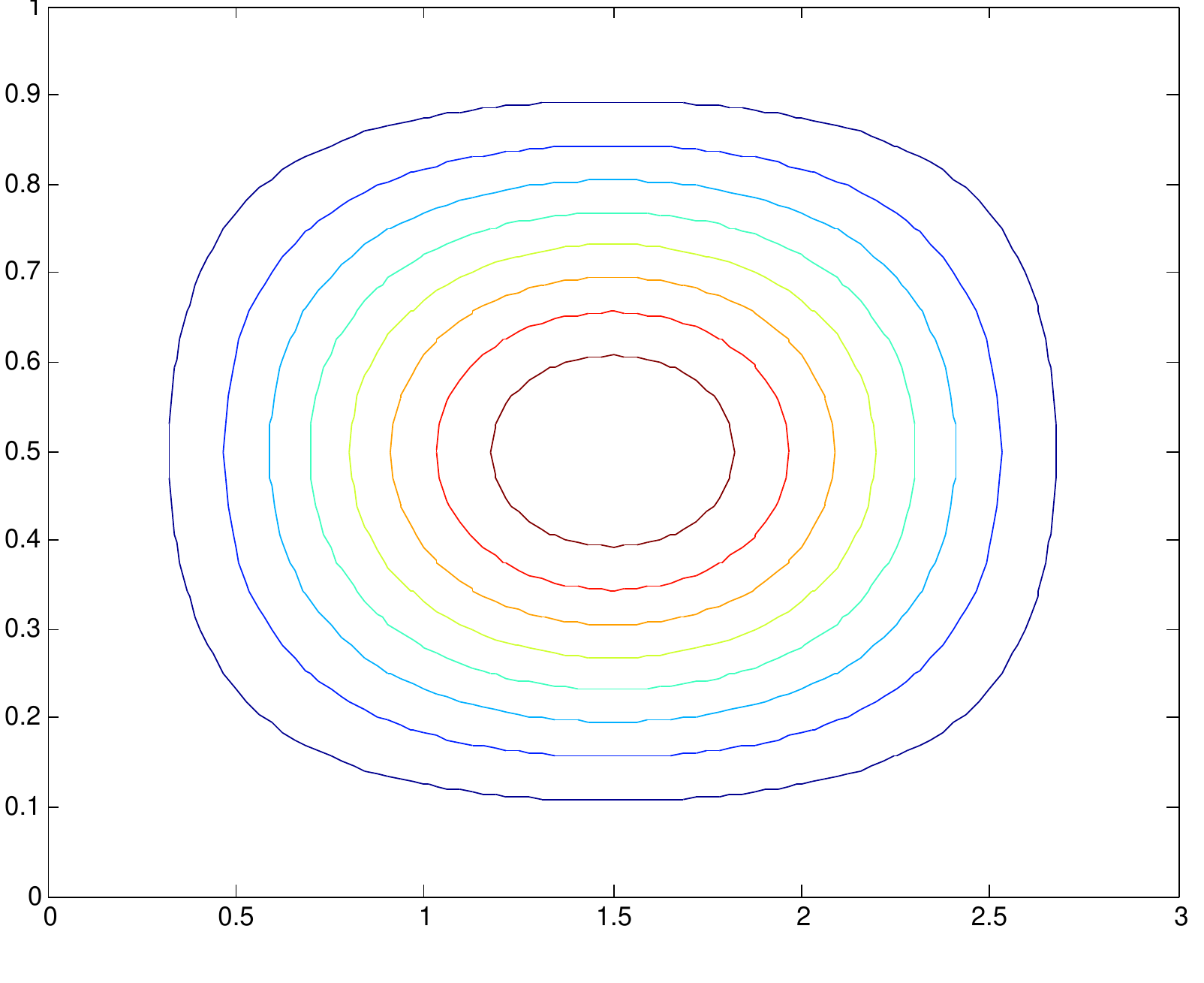}
      \label{fig:SQGEsin}
    }
    \subfigure[Test 6.]{
      \includegraphics[trim=80 200 70 215, clip=true, scale=0.38]{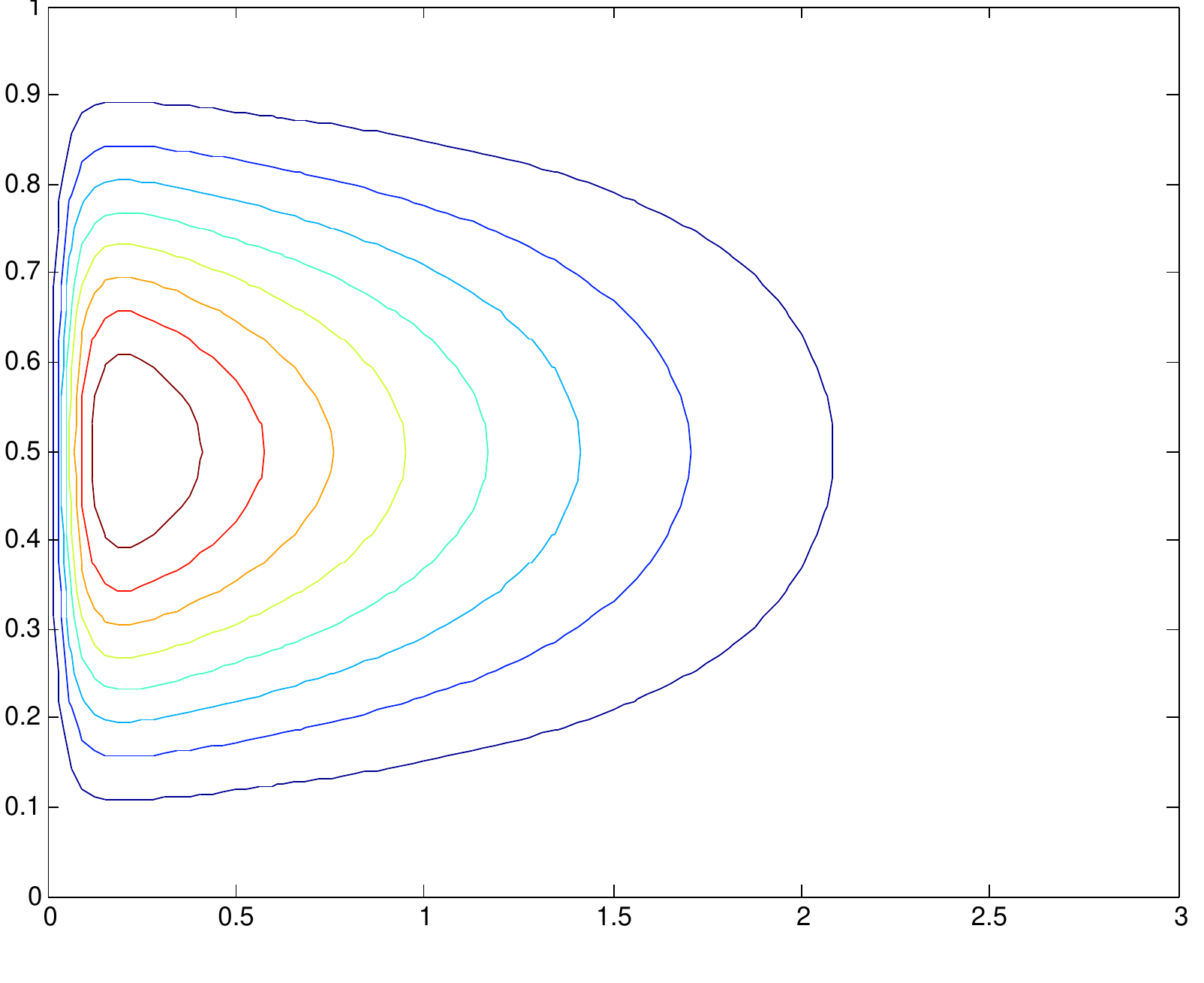}
      \label{fig:SQGEe}
    }
    \caption{QGE \eqref{qge_psi_1}: Streamlines of the approximation, $\psi^h$, on a
    mesh with $h=\frac{1}{32}$ and $28550$ DoFs.
    }
  \end{center}
\end{figure}

\tbf{Test 6:}
In this test, we take the same exact solution as that in Test 2 of \cite{Cascon}, i.e.,
  $\psi(x,y) = \left[\left(1 - \frac{x}{3}\right)\left(1-e^{-20x}\right) \sin \left( \pi y \right) \right]^2$.
The forcing term and the homogeneous boundary conditions correspond to the exact solution. 

Figure \ref{fig:SQGEe} presents the streamlines of the approximate solution obtained by using the
Argyris element on a mesh with $h=\frac{1}{32}$ and $28550$ DoFs. 
We note that Figure \ref{fig:SQGEe} resembles Figure $10$ in \cite{Myers}.
\autoref{tab:SQGEeErrors} presents the errors $e_0,\, e_1, \text{ and } e_2$ for various meshsizes $h$.
The errors in \autoref{tab:SQGEeErrors} follow the theoretical rates of convergence
predicted by the estimates \eqref{eqn:H2Error}--\eqref{eqn:L2Error} in \autoref{thm:Errors}. We see that the orders of convergence in \autoref{tab:SQGEeErrors} are close to the theoretical ones for the fine meshes, but
not as close for the coarse meshes. We attribute this to the inaccuracies at the thin boundary layer at $x=0$. 

\begin{table}
\begin{center}
{\small
\begin{tabular}{|c|c|c|c|c|c|c|c|}
  \hline
  $h$ & $DoFs$ & $e_0$ & $L_2$ order & $e_1$ & $H^1$ order & $e_2$ & $H^2$ order \\ 
  \hline
  $\nicefrac{1}{2}$ & $170$ & $0.3497$ & $-$ & $1.9$ & $-$ & $44.05$ & $-$ \\
  $\nicefrac{1}{4}$ & $550$ & $0.0302$ & $3.533$ & $0.4279$ & $2.15$ & $21.74$ & $1.019$ \\
  $\nicefrac{1}{8}$ & $1958$ & $0.001507$ & $4.324$ & $0.06085$ & $2.814$ & $5.661$ & $1.941$ \\
  $\nicefrac{1}{16}$ & $7366$ & $3.225\times 10^{-5}$ & $5.547$ & $0.004042$ & $3.912$ & $0.7379$ & $2.94$ \\
  $\nicefrac{1}{32}$ & $28550$ & $5.672\times 10^{-7}$ & $5.829$ & $0.000161$ & $4.65$ & $0.0597$ & $3.628$ \\
 \hline
\end{tabular}
}
\end{center}
\caption{QGE \eqref{qge_psi_1}, Test 6: The errors $e_0,\, e_1,\, e_2$  for various meshsizes $h$. 
} 
\label{tab:SQGEeErrors}
\end{table}

\section{Conclusions} \label{sec:Conclusions}
This paper introduced a conforming FE discretization of the streamfunction formulation of the stationary one-layer QGE based on the Argyris element. 
For this FE discretization, we proved optimal error estimates in the $H^2$, $H^1$ and $L^2$ norms.
A careful numerical investigation of the FE discretization was also performed.
To this end, the QGE as well as the linear Stommel and Stommel-Munk models (two standard simplified settings used in the geophysical fluid dynamics literature \cite{Cascon,Myers,Vallis06}) were used in the numerical tests.
Based on the numerical results from the six tests considered, we drew the following two conclusions: (i) our numerical results are close to those used in the published literature \cite{Cascon,Myers,Vallis06}; and (ii) the convergence rates of the numerical approximations do indeed follow the theoretical error estimates in Theorems \ref{thm:EnergyNorm} and \ref{thm:Errors}.
The convergence rates followed exactly the theoretical ones in the test problems where the exact solution did not display a thin boundary layer, but where somewhat lower than expected in those tests that displayed a thin western boundary layer, as expected.

We plan to extend this study in several directions, including the time-dependent QGE and the two-layer QGE.


	\bibliographystyle{siam}
	\bibliography{SQGE,../../../../../Bibliography/Traian/comprehensive_bibliography}

\begin{thebibliography}{10}

\bibitem{barcilon1988existence}
{\sc V.~Barcilon, P.~Constantin, and E.~S. Titi}, {\em Existence of solutions
  to the {S}tommel-{C}harney model of the {G}ulf {S}tream}, SIAM J. Math.
  Anal., 19 (1988), pp.~1355--1364.

\bibitem{blum1980boundary}
{\sc H.~Blum, R.~Rannacher, and R.~Leis}, {\em On the boundary value problem of
  the biharmonic operator on domains with angular corners}, Math. Methods Appl.
  Sci., 2 (1980), pp.~556--581.

\bibitem{braess2001finite}
{\sc D.~Braess}, {\em Finite elements: Theory, fast solvers, and applications
  in solid mechanics}, Cambridge University Press, 2001.

\bibitem{Cascon}
{\sc J.~M. Cascon, G.~C. Garcia, and R.~Rodriguez}, {\em A priori and a
  posteriori error analysis for a large-scale ocean circulation finite element
  model}, Comp. Meth. Appl. Mech. Eng., 192 (2003), pp.~5305--5327.

\bibitem{Cayco86}
{\sc M.~E. Cayco and R.~A. Nicolaides}, {\em Finite element technique for
  optimal pressure recovery from stream function formulation of viscous flows},
  Math. Comp., 46 (1986).

\bibitem{Ciarlet}
{\sc P.~Ciarlet}, {\em {The finite element method for elliptic problems}},
  North-Holland, 1978.

\bibitem{Cummins}
{\sc P.~F. Cummins}, {\em Inertial gyres in decaying and forced geostrophic
  turbulence}, J. Mar. Res., 50 (1992), pp.~545--566.

\bibitem{cushman2011introduction}
{\sc B.~Cushman-Roisin and J.~M. Beckers}, {\em Introduction to geophysical
  fluid dynamics: Physical and numerical aspects}, Academic Press, 2011.

\bibitem{Dijkstra05}
{\sc H.~E. Dijkstra}, {\em Nonlinear physical oceanography: A dynamical systems
  approach to the large scale ocean circulation and {E}l {N}i\~{n}o}, vol.~28,
  Springer Verlag, 2005.

\bibitem{Dominguez06}
{\sc V.~Dominguez and F.~J. Sayas}, {\em {A simple Matlab implementation of
  Argyris element}}, Tech. Rep.~25, Universidad de Zaragoza, 2006.

\bibitem{Fairag98}
{\sc F.~Fairag}, {\em A two-level finite-element discretization of the stream
  function form of the {N}avier-{S}tokes equations}, Comput. Math. Appl., 36
  (1998), pp.~117--127.

\bibitem{Fairag03}
\leavevmode\vrule height 2pt depth -1.6pt width 23pt, {\em Numerical
  computations of viscous, incompressible flow problems using a two-level
  finite element method}, SIAM J. Sci. Comp., 24 (2003), pp.~1919--1929.

\bibitem{Fix}
{\sc G.~Fix}, {\em Finite element models for ocean circulation problems}, SIAM
  J. on Appl. Math., 29 (1975), pp.~371--387.

\bibitem{fix1984mixed}
{\sc G.~Fix, M.~Gunzburger, R.~Nicolaides, and J.~Peterson}, {\em Mixed finite
  element approximations for the biharmonic equations}, Proc. 5th Internat.
  Sympos. on Finite Elements and Flow Problems (JT Oden, ed.), University of
  Texas, Austin,  (1984), pp.~281--286.

\bibitem{ghil2008climate}
{\sc M.~Ghil, M.~D. Chekroun, and E.~Simonnet}, {\em Climate dynamics and fluid
  mechanics: Natural variability and related uncertainties}, Physica D, 237
  (2008), pp.~2111--2126.

\bibitem{Girault79}
{\sc V.~Girault and P.~A. Raviart}, {\em Finite element approximation of the
  {N}avier-{S}tokes equations}, Volume 749 of Lecture Notes in Mathematics,
  Springer-Verlag, 1979.

\bibitem{Girault86}
\leavevmode\vrule height 2pt depth -1.6pt width 23pt, {\em Finite element
  methods for {N}avier-{S}tokes equations: theory and algorithms}, vol.~5 of
  Springer Series in Computational Mathematics, Springer-Verlag, 1986.

\bibitem{greatbatch2000four}
{\sc R.~J. Greatbatch and B.~T. Nadiga}, {\em Four-gyre circulation in a
  barotropic model with double-gyre wind forcing}, J. Phys. Oceanogr., 30
  (2000), pp.~1461--1471.

\bibitem{Gunzburger89}
{\sc M.~D. Gunzburger}, {\em Finite element methods for viscous incompressible
  flows}, Computer Science and Scientific Computing, Academic Press Inc, 1989.

\bibitem{gunzburger1988finite}
{\sc M.~D. Gunzburger and J.~S. Peterson}, {\em Finite-element methods for the
  streamfunction-vorticity equations: Boundary-condition treatments and
  multiply connected domains}, SIAM J. Sci. Stat. Comput., 9 (1988),
  pp.~650--668.

\bibitem{gunzburger1988onfinite}
\leavevmode\vrule height 2pt depth -1.6pt width 23pt, {\em On finite element
  approximations of the streamfunction-vorticity and velocity-vorticity
  equations}, Internat. J. Numer. Methods Fluids, 8 (1988), pp.~1229--1240.

\bibitem{Johnson}
{\sc C.~Johnson}, {\em {Numerical solution of partial differential equations by
  the finite element method}}, vol.~32, Cambridge university press New York,
  1987.

\bibitem{layton2008introduction}
{\sc W.~J. Layton}, {\em Introduction to the numerical analysis of
  incompressible viscous flows}, vol.~6, Society for Industrial and Applied
  Mathematics (SIAM), 2008.

\bibitem{leprovost1994comparison}
{\sc C.~LeProvost, C.~Bernier, and E.~Blayo}, {\em A comparison of two
  numerical methods for integrating a quasi-geostrophic multilayer model of
  ocean circulations: finite element and finite difference methods}, J. Comput.
  Phys., 110 (1994), pp.~341--359.

\bibitem{Maj03}
{\sc A.~Majda}, {\em Introduction to {PDE}s and waves for the atmosphere and
  ocean}, AMS, New York, 2003.

\bibitem{Majda}
{\sc A.~Majda and X.~Wang}, {\em Non-linear dynamics and statistical theories
  for basic geophysical flows}, Cambridge University Press, 2006.

\bibitem{mcwilliams2006fundamentals}
{\sc J.~McWilliams}, {\em {Fundamentals of geophysical fluid dynamics}},
  Cambridge University Press, 2006.

\bibitem{Medjo99}
{\sc T.~T. Medjo}, {\em Mixed formulation of the two-layer quasi-geostrophic
  equations of the ocean}, Num. Meth. P.D.E.s, 15 (1999), pp.~489--502.

\bibitem{Medjo00}
\leavevmode\vrule height 2pt depth -1.6pt width 23pt, {\em Numerical
  simulations of a two-layer quasi-geostrophic equation of the ocean}, SIAM J.
  Numer. Anal., 37 (2000), pp.~2005--2022.

\bibitem{Myers}
{\sc P.~G. Myers and A.~J. Weaver}, {\em A diagnostic barotropic finite-element
  ocean circulation model}, J. Atmos. Oceanic Technol., 12 (1995), p.~511.

\bibitem{Ped92}
{\sc J.~Pedlosky}, {\em Geophysical fluid dynamics}, Springer-Verlag,
  second~ed., 1992.

\bibitem{San11}
{\sc O.~San, A.~E. Staples, Z.~Wang, and T.~Iliescu}, {\em Approximate
  deconvolution large eddy simulation of a barotropic ocean circulation model},
  Ocean Modelling, 40 (2011), pp.~120--132.

\bibitem{stevens1982finite}
{\sc W.~N.~R. Stevens}, {\em {Finite element, stream function--vorticity
  solution of steady laminar natural convection}}, Int. J. Num. Meth. Fluids, 2
  (1982), pp.~349--366.

\bibitem{temam2001navier}
{\sc R.~Temam}, {\em Navier-Stokes equations: theory and numerical analysis},
  vol.~2, American Mathematical Society, 2001.

\bibitem{Vallis06}
{\sc G.~K. Vallis}, {\em Atmosphere and ocean fluid dynamics: Fundamentals and
  large-scale circulation}, Cambridge University Press, 2006.

\bibitem{wang1994emergence}
{\sc J.~Wang and G.~K. Vallis}, {\em Emergence of {F}ofonoff states in inviscid
  and viscous ocean circulation models}, J. Mar. Res., 52 (1994), pp.~83--127.

\bibitem{wolansky1988existence}
{\sc G.~Wolansky}, {\em Existence, uniqueness, and stability of stationary
  barotropic flow with forcing and dissipation}, Comm. Pure Appl. Math., 41
  (1988), pp.~19--46.

\end{thebibliography}
\end{document}